\documentclass[11pt]{article}
\usepackage{authblk}
\usepackage{verbatim}
\usepackage{amssymb,latexsym,amsmath,graphicx,amsfonts,amsthm,xy,eucal,mathrsfs,fullpage, epsfig}
\usepackage{amscd,subfigure,color,setspace,cite,mathtools}
\usepackage[font=footnotesize]{caption}
\usepackage{float}
\usepackage{caption}
\newcommand\bovermat[2]{%
\makebox[0pt][l]{$\smash{\overbrace{\phantom{%
\begin{matrix}#2\end{matrix}}}^{\text{#1}}}$}#2}

\usepackage{epstopdf}
\usepackage[english]{babel}
\usepackage{babel}
\usepackage{fontenc}
\usepackage[colorlinks=true,citecolor=blue]{hyperref}
\usepackage{float}
\usepackage{lineno}
\usepackage{pgf,tikz}
\usepackage{array}
\usepackage{multirow}
\makeatletter
\renewcommand*\env@matrix[1][\arraystretch]{%
  \edef\arraystretch{#1}%
  \hskip -\arraycolsep
  \let\@ifnextchar\new@ifnextchar
  \array{*\c@MaxMatrixCols c}}
\makeatother
\usepackage{float}
\restylefloat{table}
\usetikzlibrary{arrows}
\usetikzlibrary[patterns]
\usetikzlibrary{decorations.pathmorphing}
 \usetikzlibrary{plotmarks}
        \usetikzlibrary{intersections}
        \usetikzlibrary{shadings}
        \usetikzlibrary{calc}
\newlength{\hatchspread}
\newlength{\hatchthickness}
\newlength{\hatchshift}
\newcommand{\hatchcolor}{}
\tikzset{hatchspread/.code={\setlength{\hatchspread}{#1}},
         hatchthickness/.code={\setlength{\hatchthickness}{#1}},
         hatchshift/.code={\setlength{\hatchshift}{#1}},% must be >= 0
         hatchcolor/.code={\renewcommand{\hatchcolor}{#1}}}
\tikzset{hatchspread=3pt,
         hatchthickness=0.4pt,
         hatchshift=1pt,% must be >= 0
         hatchcolor=black}
\pgfdeclarepatternformonly[\hatchspread,\hatchthickness,\hatchshift,\hatchcolor]% variables
   {custom north west lines}% name
   {\pgfqpoint{\dimexpr-2\hatchthickness}{\dimexpr-2\hatchthickness}}% lower left corner
   {\pgfqpoint{\dimexpr\hatchspread+2\hatchthickness}{\dimexpr\hatchspread+2\hatchthickness}}% upper right corner
   {\pgfqpoint{\dimexpr\hatchspread}{\dimexpr\hatchspread}}% tile size
   {% shape description
    \pgfsetlinewidth{\hatchthickness}
    \pgfpathmoveto{\pgfqpoint{0pt}{\dimexpr\hatchspread+\hatchshift}}
    \pgfpathlineto{\pgfqpoint{\dimexpr\hatchspread+0.15pt+\hatchshift}{-0.15pt}}
    \ifdim \hatchshift > 0pt
      \pgfpathmoveto{\pgfqpoint{0pt}{\hatchshift}}
      \pgfpathlineto{\pgfqpoint{\dimexpr0.15pt+\hatchshift}{-0.15pt}}
    \fi
    \pgfsetstrokecolor{\hatchcolor}
    \pgfusepath{stroke}
   }
\usetikzlibrary{arrows}
\usetikzlibrary[patterns]
\numberwithin{equation}{section}
\newtheorem{theorem}{Theorem}[section]

\newtheorem{lemma}[theorem]{Lemma}
\newtheorem{example}[theorem]{Example}
\newtheorem{proposition}[theorem]{Proposition}

\newtheorem{remark}{Remark}[section]

\begin{document}\title{\textbf{Increasing delay as a strategy to prove stability}}
\author[1,2]{ Ziyad AlSharawi\thanks{Corresponding author: zsharawi@aus.edu. This work was done while the first author was on sabbatical leave from the American University of Sharjah.}}
\author[2]{ Jose S. C\'anovas}
\affil[1]{\small American University of Sharjah, P. O. Box 26666, University City, Sharjah, UAE}
\affil[2]{\small Universidad Politécnica de Cartagena, Paseode Alfonso XIII 30203, Cartagena, Murcia, Spain}
\date{\today}
\maketitle

\begin{abstract}

We consider difference equations of the form $x_{n+1}=F_0(x_n,\ldots,x_{n-k+1}),$ and increase the delay through a process of successive substitutions to obtain a sequence of systems $y_{n+1}=F_j(x_{n-j},\ldots,x_{n-k-j+1}),\; j=0,1,\ldots$. We call this process \emph{the expansion strategy} and use it to establish novel results that enable us to prove stability. When the map $F_0$ is sufficiently smooth and has a hyperbolic fixed point, we show the fixed point is locally asymptotically stable if and only if  $\|\nabla F_j\|_1<1$ for some finite number $j$. Our local stability results complement recent results obtained on Schur stability, and they can provide an alternative to the highly acclaimed Jury's algorithm. Also, we show the effectiveness of the expansion strategy in obtaining global stability results. Global stability results are obtained by integrating the expansion strategy with the embedding technique. Finally, we give illustrative examples to show the results' practical applicability across various discrete-time dynamical systems.
\end{abstract}

\section{ Introduction}
Introducing a time delay in first-order discrete-time systems of the form $X_{n+1}=F(X_n)$ results in the higher-dimensional system $X_{n+1}=F(X_{n-k}).$ This type of time delay has a foreseeable impact, and it is known to cause oscillations that give freedom to the existence of periodic solutions\cite{Al-An-El2008, Di-Gi2000, Ba-Li2003}.  However, we are interested in investigating a more peculiar type of delay increase in this paper.  We consider the scalar difference equation 
\begin{equation}\label{Eq-F0}
x_{n+1}=F_0(x_n,x_{n-1},\ldots,x_{n-k+1})\quad \text{for}\quad k>1,
\end{equation}
where $F_0:V^k\to V$ and $V$ can be $\mathbb{R},$ $\mathbb{R}_+=[0,\infty)$ or a compact interval in  $\mathbb{R}_+$. We assume $F_0$ is sufficiently smooth,  and expand the time delay by employing a forward substitution to obtain
\begin{equation}\label{Eq-TheExpansion}
\begin{split}
y_{n+1}=&F_0(F_0(y_{n-1},\ldots,y_{n-k}),y_{n-1},\ldots,y_{n-k+1})\\
=&F_1(y_{n-1},\ldots,y_{n-k+1},y_{n-k}).
\end{split}
\end{equation}
   By focusing on the initial conditions of  Eqs. \eqref{Eq-F0} and \eqref{Eq-TheExpansion}, it is easy to observe that $F_1$ gives a new system with a higher dimension. More specifically, if $F_0:V^k\to V,$ then $F_1:V^{k+1}\to V,$ where $V^{k+1}=V^k\times V$. Our main concern in this paper is to reveal the relationship between the dynamics of Eq. \eqref{Eq-F0} and the dynamics of its expansion in Eq. \eqref{Eq-TheExpansion}, then generalize this process through successive substitutions.  It was observed in \cite{Al-Ca-Ka2025, El-Lo2008,Ca-La2008} that this substitution (or delay expansion) may offer some advantages in establishing stability or boundedness. However, to the best of our knowledge, no established theory presently addresses this notion as a cohesive strategy. This observation inspires us to examine the relationship between the dynamics of the original $F_0$-system and the expanded systems.   In particular, stability in the expanded systems will be investigated and utilized to prove stability in the original $F_0$-system. We refer to this technique as \emph{the expansion strategy}.
\\

Throughout this paper, the map $F_0$ is considered sufficiently smooth. Since it is customary in applications to confine the domain of $F_0$ to the positive cone $V^k=\mathbb{R}_+^k$, we consider it so unless otherwise specified. The eigenvalues of a square matrix $A$ are denoted by the spectrum of $A$ $ (Spec(A)).$ We adopt a strict local asymptotic stability definition (LAS) and call a fixed point $\bar x$ of $F_0$ in Eq.  \eqref{Eq-F0} LAS if the spectrum of the Jacobian matrix at the fixed point is contained in the open unit disk of the complex plane, i.e., $\mathbb{D}=\{(x,y):\; x^2+y^2<1\}.$ It is worth mentioning that we use $\bar x$ as a scalar value when we talk about fixed points of $F_j,$ while as a vector with $k+j$ components when we talk about equilibrium solutions of an $F_j$-system.  
 A unique fixed point of $F_0$ is called globally attracting if each orbit of $F_0$ converges to the fixed point. In this case, LAS and globally attracting imply global asymptotic stability (GAS).   On the other hand, $\bar x$  of Eq. \eqref{Eq-F0} is called unstable if at least one eigenvalue of the Jacobian matrix of $F_0$ at $\bar x$ is located out of the closed unit disk $\bar{\mathbb{D}}$. If some eigenvalues are in $\mathbb{D}$ while others are on the boundary of $\mathbb{D},$ the fixed point is called nonhyperbolic, and its stability is beyond our interest in this paper.  For a single-variable polynomial $p(x),$ we use $\|p\|_{\ell_1}$ to denote the $\ell_1$-norm, i.e., the absolute sum of the coefficients. Similarly, the norm $\|U\|_1$ denotes the absolute sum of the coordinates whenever $U\in\mathbb{R}^n.$
\\

The second iteration of a system was used in the literature to show boundedness \cite{Ca-La2008}, and the two-dimensional Ricker model is a simple example. Indeed, boundedness of orbits is not obvious in $x_{n+1}=F_0(x_n,x_{n-1})=x_n\exp(b-x_{n-1}),$ but it is obvious in 
\begin{align*}
y_{n+1}=&F_1(y_{n-1},y_{n-2})\\
=&y_{n-1}\exp(b-y_{n-1})\exp(b-y_{n-2})\\
\leq& e^{2b-1},\;\forall y_0,y_{-1},y_{n-2}\geq 0.
\end{align*}
Similar ideas were used in \cite{El-Lo2008,Lo2010} to find dominating one-dimensional maps that can be used to reveal the global attractor of some forms of $F_0.$ The second iterate of $x_{n+1}=x_nf(x_{n-1})+h_n$ was used in \cite{Al2013} to show boundedness, invariant domains, and prove the global attractor based on a $\limsup$ and $\liminf$ argument.
In a recent paper, the authors in \cite{Al-Ca-Ka2025} focused on Schur stability and found that combining the embedding technique with successive expansions in linear systems gives favorable sufficient conditions for Schur stability. This paper formalizes this approach further and complements the results obtained in \cite{Al-Ca-Ka2025} by establishing necessary and sufficient conditions.
\\

This paper is structured as follows: In section two, we clarify the notion of the expansion strategy. Section three considers the linear case, uses the expansion strategy to increase the delay, and establishes the relationship between the original and expanded systems. This develops and complements the algorithm established in \cite{Al-Ca-Ka2025}, which gives an alternative to Jury's algorithm. In section four, we consider the case of nonlinear systems. A linearization process allows us to investigate local stability based on the results of section three. Furthermore, we integrate the embedding technique and expansion strategy to obtain new results in global stability. In section five, we present several applications to show the power and effectiveness of the obtained results. Finally, we close this paper with a conclusion section.

%%%%%%%%%%%%%%%%%%%%%%%%%%%%%%%%%%%%%%%%%%%%%%%%%%%%%%%%%%%%%%%%%
\section{The expansion strategy}\label{Sec-Expansion}
%%%%%%%%%%%%%%%%%%%%%%%%%%%%%%%%%%%%%%%%%%%%%%%%%%%%%%%%%%%%%%%%%
When successive expansions are done as in Eq. \eqref{Eq-TheExpansion}, we obtain a sequence of associated systems of the form
\begin{equation}\label{Eq-TheGeneralExpansion}
y_{n+1}=F_m(y_{n-m},y_{n-m-1},\ldots,y_{n-m-k+1}),
\end{equation}
where $m=0$ represents the original system in Eq. \eqref{Eq-F0}. Obviously, $F_0$ requires the initial conditions $x_{-k+1},\ldots,x_0,$ while for $m\geq 1,$ $F_m$ requires the initial conditions 
$$y_{-m-k+1},\ldots,y_{-m},\ldots,y_0.$$
We clarify the relationship between the solutions of Eqs. \eqref{Eq-F0} and \eqref{Eq-TheGeneralExpansion}. The domain of $F_0$ in Eq. \eqref{Eq-F0} is $\mathbb{R}_+^k$ while the domain of $F_m$ in Eq. \eqref{Eq-TheGeneralExpansion} is $\mathbb{R}_+^{k+m}$; however, the difference in dimension between the two equations can be ignored when we consider solutions in scalar form rather than vector form.  First, we clarify fixed points. Let $ x$ be a fixed point of $F_0.$ Because
\begin{equation}\label{Eq-PreservingFixedPoints}
F_{m+1}(x,x,\ldots,x)=F_m(F_0(x,\ldots,x),x,\ldots,x)=F_m(x,x,\ldots,x),
\end{equation}
it is obvious that a fixed point of $F_0$ is a fixed point of $F_{m}$ for all $m$. But the converse is not necessarily true, as we clarify by considering the nonlinear example
\begin{equation}\label{Eq-ExampleOfCreatingNewFixedPoints}
x_{n+1}=F_0(x_n,x_{n-1})=x_{n-1}\exp{(2-x_n)}+1.
\end{equation}
For $m=1,$ we have $x_{n+1}=F_0(F_0(x_{n-1},x_{n-2}),x_{n-1})$. The fixed points are solutions to
$$x=F_0(y,x)=x\exp{(2-y)}+1,\quad\text{where}\quad y=F_0(x,x)=x\exp{(2-x)}+1.$$
This gives us $\bar x_1\approx 2.509$ or $\bar x_2\approx 1.237.$ $\bar x_1$ is a fixed point of both $F_0$ and $F_1,$ while $\bar x_2$ is a fixed point of $F_1$ only. Next, we clarify solutions in general,  and give the result in the next proposition.

\begin{proposition}\label{Pr-ChangeInFixedPoints}
A solution of Eq. \eqref{Eq-F0} can be made a solution of Eq. \eqref{Eq-TheGeneralExpansion}, but the converse is not necessarily true. In particular, a fixed point of $F_0$ is a fixed point of $F_{m}$ for all $m\geq 0,$ but the converse is not necessarily true. 
\end{proposition}
\begin{proof}
For each $m\geq 0$, we consider a solution of the $F_{0}$-equation and show that it can be a solution of the $F_{m}$-equation. Let
$$S_0:=\{\underbrace{x_{-k+1},\ldots,x_{-1}, x_0}_{\text{initial conditions}},x_1,\ldots\}$$
be a solution of the $F_0$-equation. For the $F_1$-equation, consider the initial conditions $y_{-k},y_{-k+1},\ldots,y_0,$
where $y_{j-1}=x_j$ for all $j=-k+1$ to $1.$ In this case, we obtain 
$$
  y_1=F_1(y_{-1},\ldots,y_{-k})  = F_1(x_0,x_{-1},\ldots,x_{-k+1})$$
  and consequently, 
  $$y_1=F_0(F_0(x_0,\ldots,x_{-k+1}),x_0,x_{-1},\ldots,x_{-k+2})=F_0(x_1,x_0,\ldots,x_{-k+2})=x_2.$$
By induction on $j$, we obtain $y_j=x_{j+1}$ for all $j>1.$ Observe that $y_0$ is supposed to be a free initial condition in the $F_1$-equation, but we forced $y_0$ to be $x_1.$ Similarly, we do an induction on $m,$ assume $S_0$ is a solution of the $F_m$-equation, then show that it can be a solution of $F_{m+1}$-equation. Indeed, we define the initial conditions $z_{-m-k+j}=x_{-k+1+j}$ for $j=0,\ldots,m+k,$ then
\begin{align*}
  z_1=&F_{m+1}(z_{-m-1},\ldots,z_{-m-k})  \\
  =& F_{m+1}(x_0,x_{-1},\ldots,x_{-k+1})\\
  =&F_m(F_0(x_0,\ldots,x_{-k+1}),x_0,x_{-1},\ldots,x_{-k+2})\\
  =&F_m(x_1,x_0,\ldots,x_{-k+2})=x_{m+2}.
\end{align*}
Similarly, $z_n=x_{n+m+1}$ for all $n=2,\ldots . $  Finally, it is obvious that a fixed point of $F_0$ is a fixed point of $F_{m}$ for all $m\geq 0,$ and the converse is not necessarily true as we clarified in the example of Eq. \eqref{Eq-ExampleOfCreatingNewFixedPoints}. 
\end{proof}
 
Up to this end, we clarified that a solution of Eq. \eqref{Eq-F0} can be a solution of Eq. \eqref{Eq-TheGeneralExpansion}, and Eq. \eqref{Eq-PreservingFixedPoints} emphasizes this fact in particular for fixed points. However, the example given in Eq. \eqref{Eq-ExampleOfCreatingNewFixedPoints} shows that the expansion may create new fixed points.  The following result ensures that the fixed points of $F_m$ and its consequent expansion are the same.   

\begin{proposition}\label{Pr-Equilibrium}
Suppose each one of the maps $F_m$ obtained in the expanded systems of Eq. \eqref{Eq-TheGeneralExpansion} is increasing in its first argument, then $F_{m}$ and $F_{m+1}$ have the same fixed points.  
\end{proposition}
\begin{proof}
Suppose $F_{m+1}$ has a fixed point different from the fixed points of $F_m,$ say $F_{m+1}(x,\ldots,x)=x.$ In this case, we obtain
$$x=F_m(y,x,\ldots,x),\quad\text{where}\quad y=F_0(x,x,\ldots,x).$$
Write
\begin{align*}
 y-x=&-\left(F_m(y,x,\ldots,x)-F_0(x,x,\ldots,x)\right)\\
 =& -\sum_{j=0}^m \left(F_j(y,x,\ldots,x)-F_j(x,x,\ldots,x)\right)\\
 -1=& \sum_{j=0}^m \frac{F_j(y,x,\ldots,x)-F_j(x,x,\ldots,x)}{y-x}.
\end{align*}
This is a contradiction since the right-hand side is positive. 
\end{proof}
%%%%%%%%%%%%%%%%%%%%%%%%%%%%%%%%%%%%%%%%%%%%%%%%%%%%%%%%%%%%%%%%%
\section{The linear case}\label{Sec-Linear}
%%%%%%%%%%%%%%%%%%%%%%%%%%%%%%%%%%%%%%%%%%%%%%%%%%%%%%%%%%%%%%%%%
In this section, we utilize the expansion strategy to complement the results obtained in \cite{Al-Ca-Ka2025}. Consider the $k$-dimensional linear equation
\begin{equation}\label{Eq-Linear1}
x_{n+1}=F_0(x_n,\ldots,x_{n-k+1})=\sum_{j=0}^{k-1}a_jx_{n-j},\quad k>1,\;a_j\in \mathbb{R}\;\text{and}\; a_0\neq 0.
\end{equation}
After substituting in place of $x_n,$ we obtain the one-step expansion
\begin{align*}
y_{n+1}=a_0\sum_{j=0}^{k-1}a_jy_{n-j-1}+\sum_{j=1}^{k-1}a_jy_{n-j},
\end{align*}
which is the $(k+1)$-dimensional system
\begin{equation}\label{Eq-Linear2}
\begin{split}
y_{n+1}=&F_1(y_{n-1},\ldots,y_{n-k+1},y_{n-k})\\
=&a_0a_{k-1}y_{n-k}+\sum_{j=0}^{k-2}(a_0a_j+a_{j+1})y_{n-j-1}.
\end{split}
\end{equation}
Repeat the process $m$-times, where $m=0$ means Eq. \eqref{Eq-Linear1}. By induction, we obtain
\begin{equation}\label{Eq-Linear3}
\begin{split}
z_{n+1}=&F_m(z_{n-m},\ldots,z_{n-k-m+1})\\
=&\sum_{j=0}^{k-1}b_{m,j}z_{n-m-j},
\end{split}
\end{equation}
where $b_{0,j}=a_j,\; j=0,1,\ldots,k-1$ and
\begin{equation}\label{Eq-The-b-System}
\begin{cases}
b_{m+1,j}=&b_{0,j}b_{m,0}+b_{m,j+1},\quad j=0,\ldots,k-1\\
b_{m+1,k-1}=&b_{m,0}b_{0,k-1}
\end{cases}
\end{equation}
for all $m\in\mathbb{N}=\{0,1,\ldots\}.$ Based on this, we obtain the following crucial lemma: 
\begin{lemma}\label{Lem-Crucial}
Let the column vector $V_0=[a_0,a_1,\ldots,a_{k-1}]^t$ denote the coefficients of $F_0$ in Eq. \eqref{Eq-Linear1}. The coefficients of $F_m$ in the expanded system \eqref{Eq-Linear3} are obtained by the column vector $V_m,$ where 
\begin{equation}\label{Eq-Vm}
V_{m}=(J_0^t)^{m}V_0,\quad m=1,2,\ldots.
\end{equation}
\end{lemma}
\begin{proof}
Write the coefficients obtained in System \eqref{Eq-The-b-System} as a linear system. We obtain
$$\left[
    \begin{array}{c}
      b_{m+1,0} \\
      b_{m+1,1} \\
      \vdots \\
      b_{m+1,k-2} \\
      b_{m+1,k-1} \\
    \end{array}
  \right]=J_0^t\left[
    \begin{array}{c}
      b_{m,0} \\
      b_{m,1} \\
      \vdots \\
      b_{m,k-2} \\
      b_{m,k-1} \\
    \end{array}
  \right],
 $$
 where $J_0$ is the Jacobian matrix of $F_0$ and $J_0^t$ is the transpose of $J_0.$
Therefore, we have 
$$V_{m+1}=J_0^tV_m,\quad \text{where}\quad  
V_m=[b_{m,0},b_{m,1},\ldots,b_{m,k-1}]^t\;\; \text{for all}\;\; m\in\mathbb{N}.$$
The solution of this equation gives Eq. \eqref{Eq-Vm}, and the proof is complete. 
\end{proof}
Up to this end, it becomes computationally simple and straightforward that the coefficients of $F_m$ in Eq. \eqref{Eq-Linear3} are captured in the vector $V_m,$ which can be computed based on Eq. \eqref{Eq-Vm}. 
\\

\begin{remark}
We considered $ a_0\neq 0$ in our analysis and discussion for convenience. However, if $a_0=0$, then there is no need to establish a new process. For instance, when $a_0=0$ and $a_1\neq 0,$ the same process works, but in this case, $F_0=F_1$ and changes start to affect $F_2.$ In general, if the first non-zero coefficient is $a_i,$ for some $ 0<i<k-1,$ then $F_0=F_1=\cdots=F_i,$ and the expansion effect starts to take place at $F_{i+1}.$
\end{remark}

As confirmed in Proposition \ref{Pr-ChangeInFixedPoints}, initial conditions of Eq. \eqref{Eq-Linear1} must be $x_{-k+1},\ldots,x_{-1},x_0,$ while  initial conditions of Eq. \eqref{Eq-Linear3} must be $z_{-m-k+1},\ldots,z_{-1},z_0.$ Therefore, there is more freedom in assigning the initial conditions of Eq. \eqref{Eq-Linear3}. Now, we proceed to investigate the effect of the expansion process on the stability of the zero equilibrium.
%%%%%%%%%%%%%%%%%%%%%%%%%%%%%%%%%%%%%%%%%%%%%%%%%%%%%%%%%%%
\subsection{Stability and constraints on coefficients}
%%%%%%%%%%%%%%%%%%%%%%%%%%%%%%%%%%%%%%%%%%%%%%%%%%%%%%%%%%%
Denote the Jacobian matrix of the linear system in Eq. \eqref{Eq-Linear3} by $J_m$ and its spectrum by $Spec(J_m).$ At $m=0,$ it is clear that we can write $J_0=[a_{i,j}]$, where $a_{1,j}=a_{j-1}$ for all $j=1,\ldots,k,$ $a_{i+1,i}=1$ for $i=1,\ldots,k-1$, and $a_{i,j}=0$ otherwise.  The matrix $J_0 $ is, in fact, a form of the companion matrices of the characteristic polynomial (we denote it by ``the companion matrix").  The zero equilibrium of System \eqref{Eq-Linear1} is globally attracting if $Spec(J_0)$ is contained in the open unit disk $\mathbb{D}.$ Also,   observe that if we write $J_1=[b_{i,j}]$, then $b_{1,j}=a_0a_{j-2}+a_{j-1},$ for all $j=2,\ldots,k-1,$ $b_{1,k}=a_0a_{k-1},$ $b_{i+1,i}=1$ for $i=1,\ldots,k$, and $b_{i,j}=0$ otherwise. In general, the size of $J_m$ expands based on $m$, and the relationship between $J_m$ and the coefficients of the linear equation obtained by $F_m$ can be obtained from Lemma \ref{Lem-Crucial}. For the reader's convenience, we show the structure of $J_m$ below.

$$ J_m=
\begin{matrix}\begin{bmatrix}[1.3]
\bovermat{m zeros}{0 & 0 &0& \cdots & 0 &} \bovermat{$V_m^t$}{ b_{m,0}&b_{m,1}&\ldots&b_{m,k-1} } \\
%0&0&0&\cdots&0&b_{m,0}&b_{m,1}&\ldots&b_{m,k-1}\\
1&0&0&\cdots&0&0&0&\ldots&0\\
0&1&0&\cdots&0&0&0&\ldots&0\\
\vdots&\ddots&\ddots&\ddots&\ddots&\ddots&\ddots&\ddots&\vdots\\
0&0&0&0&0&\cdots&1&0&0\\
0&0&0&0&0&\cdots&0&1&0\\
\end{bmatrix}_{(m+k)\times(m+k)}.
\end{matrix}$$
By calibrating the initial conditions in System  \eqref{Eq-Linear2}, Proposition \ref{Pr-ChangeInFixedPoints} confirms that every solution of System \eqref{Eq-Linear1} can be a solution of System  \eqref{Eq-Linear2}. So, it is natural to expect $Spec(J_0)\subseteq Spec(J_1).$ Indeed, this can be clarified algebraically, and is given in the following proposition.

\begin{proposition}\label{Pr-Spectrum}
Let $J_0$ and $J_1$ be the Jacobian matrices of $F_0$ in System \eqref{Eq-Linear1} and $F_1$ in System \eqref{Eq-Linear1}, respectively. If $Spec(J_0)=\{\lambda_1,\ldots,\lambda_k\},$ then $\displaystyle{Spec(J_1)=Spec(J_0)\cup\{-a_0\}}.$
\end{proposition}
\begin{proof}
Suppose $\lambda\in Spec(J_0),$ then $\det(\lambda I-J_0)=0.$ Based on our notation in $J_0$ and $J_1,$ we obtain
$$
\det(\lambda I-J_0)=\lambda^k-\sum_{j=0}^{k-1}a_j\lambda^{k-j-1}$$
and 
$$\det(\lambda I-J_1)= \lambda^{k+1}-\sum_{j=0}^{k-2}(a_0a_j+a_{j+1})\lambda^{k-j-1}-a_0a_{k-1}.$$
Therefore, we have 
\begin{align*}
\det(\lambda I-J_1)=& \lambda^{k+1}-a_0\sum_{j=0}^{k-1}a_j\lambda^{k-j-1}-\sum_{j=0}^{k-2}a_{j+1}\lambda^{k-j-1}\\
=&(\lambda+a_0)\det(\lambda I-J_0).
\end{align*}
This clarifies the relationship between the spectra of $J_0$ and $J_1.$ Note that 
the trace of $J_1$ is an alternative way to observe that the new eigenvalue is $\lambda_{k+1}=-a_0.$ Indeed, we have  $Trace(J_0)=\sum_{j=1}^k\lambda_j$ while $Trace(J_1)=0.$ 
\end{proof}
Proposition \ref{Pr-Spectrum} tells us that it is possible to have the zero fixed point of $F_0$ in System \eqref{Eq-Linear1} stable, while the zero fixed point of $F_1$ in System \eqref{Eq-Linear2} is unstable. Example \ref{Ex-Easy1} illustrates this fact.
\begin{example}\label{Ex-Easy1}
Consider Eq. \eqref{Eq-Linear1} with $k=2$ and $a_0=\frac{5}{4},a_1=-\frac{3}{8}.$ This gives us $\lambda_1=\frac{1}{2}$ and $\lambda_2=\frac{3}{4}.$ On the other hand, Eq. \eqref{Eq-Linear2} becomes
$$y_{n+1}=\frac{19}{16}y_{n-1}-\frac{15}{32}y_{n-2}.$$
The eigenvalues are  $\lambda_1=\frac{1}{2},$  $\lambda_2=\frac{3}{4}$ and $\lambda_3=-\frac{5}{4}$. Since $\lambda_1,\lambda_2\in\mathbb{D},$ but $\lambda_3\notin \overline{\mathbb{D}},$ the fixed point $\bar x =0$ in System \eqref{Eq-Linear1} is globally attracting, while $\bar y=0$ is unstable for System \eqref{Eq-Linear2}. 
\end{example}
Proposition \ref{Pr-Spectrum} and Example \ref{Ex-Easy1} show that to have stability in both systems, it is necessary to have the absolute sum of the spectrum less than one, i.e., $-1<\sum\lambda_i<1.$ This takes us to the question: what happens when we increase the delay further through the same expansion process?  To address this question, define the polynomials $p_m$ and $q_m$ as
\begin{align}\label{Eq-Def-Pm}
\begin{split}
p_m(x)=&x^{k+m}-\sum_{j=0}^{k-1}b_{m,j}x^{k-j-1},\\
q_m(x)=&x^m+\sum_{i=0}^{m-1}b_{i,0}x^{m-i-1} .
\end{split}
\end{align}
By convention,  we consider $q_m(x)=1$ when $m=0.$  Recall that $J_m$ is the Jacobian matrix of System \eqref{Eq-Linear3}, and observe that $p_m$ is the characteristic polynomial of $J_m.$ The following relationship between $p_m$ and $q_m$ is needed in the sequel.
\begin{lemma}\label{Lem-Pm}
Consider $p_m$ and $q_m$ as defined in Eqs. \eqref{Eq-Def-Pm}. We have $ p_m(x)=p_0(x)q_m(x)$ for all $m\in \mathbb{N}.$ 
\end{lemma}
\begin{proof}
We prove this relationship by induction on $m.$ It is true by convention at $m=0,$ and Proposition \ref{Pr-Spectrum} gives the verification at $m=1.$ We assume the relationship is valid for $m,$ and we verify it for $m+1.$ We have
\begin{align*}
p_0(x)q_{m+1}(x)=& p_0(x)\left(x^{m+1}+\sum_{i=0}^{m}b_{i,0}x^{m-i}\right)\\
=&p_0(x)\left(x^{m+1}+x\sum_{i=0}^{m-1}b_{i,0}x^{m-i-1}+b_{m,0}\right)\\
=&p_0(x)\left(xq_m(x)+b_{m,0}\right).
\end{align*}
Now, by the induction argument, we have $p_m(x)=p_0(x)q_m(x).$ Therefore, we proceed as
\begin{align*}
p_0(x)q_{m+1}(x)=&xp_m(x)+b_{m,0}p_0(x)\\
=&x\left(x^{k+m}-\sum_{j=0}^{k-1}b_{m,j}x^{k-j-1}\right) +\left(x^k-\sum_{j=0}^{k-1}a_{j}x^{k-j-1}\right)b_{m,0}\\
=&x^{k+m+1}-\sum_{j=0}^{k-2}\left(b_{m,j-1}+a_jb_{m,0}\right)x^{k-j-1} -a_{k-1}b_{m,0}\\
=&p_{m+1}(x).
\end{align*}
\end{proof}
Lemma \ref{Lem-Pm} describes the spectrum of $J_m$ as the union of the spectrum of $J_0$ and the zeros of $q_m.$ For instance, we already found that  $Spec(J_1)=Spec(J_0)\cup\{-a_0\}.$ Also, it is easy to check that   $Spec(J_2)$ is the union of $Spec(J_0)$ and the zeros of $q_2(x)=x^2+a_0x+a_0^2+a_1.$ 
\\

Note that the algorithm provided in \cite{Al-Ca-Ka2025} says if the $\ell_1$-norm of $p_m$ is less than two for some $m\in \mathbb{N},$ then the equilibrium solution is globally attracting in both systems \eqref{Eq-Linear3} and \eqref{Eq-Linear1}. This was a sufficient condition. However, our results here enable us to provide sufficient and necessary conditions. Before we give the main results of this section, we illustrate this process by revisiting Example \ref{Ex-Easy1}. We do the process for up to $m=4$ and summarize the computations in Table \ref{Tab-1} and Table \ref{Tab-2}.
\begin{table}[H]
\caption{This table shows the obtained linear systems and the $\ell_1$-norm of their characteristic polynomials after applying the expansion technique for up to $m=4.$ }\label{Tab-1}
\centering
\begin{tabular}{lclcc}
\hline\hline\\ [-1.5ex]
 $m$ &&System&&$\|p_m\|_{\ell_1}$ \\ [0.5ex]
\hline
&&&&\\
 $0$& &$x_{n+1}=\frac{5}{4}x_n-\frac{3}{8}x_{n-1}$&& $2.625$\\[2ex]
$1$&& $x_{n+1}=\frac{19}{16}x_{n-1}-\frac{15}{32}x_{n-2}$&&  $2.656$ \\[2ex]
$2$&& $x_{n+1}=\frac{65}{64}x_{n-2}-\frac{57}{128}x_{n-3}$&&  $2.461$ \\[2ex]
$3$&& $x_{n+1}=\frac{211}{256}x_{n-3}-\frac{195}{512}x_{n-4}$&&  $2.205$ \\[2ex]
$4$&& $x_{n+1}=\frac{665}{1024}x_{n-4}-\frac{633}{2048}x_{n-5}$&&  $1.958$ \\[2ex]
\hline
\end{tabular}
\end{table}

Now, we give the main results of this section.

\begin{theorem}\label{Th-Stability1}
Consider the linear systems in Eq. \eqref{Eq-Linear1} and Eq. \eqref{Eq-Linear3}. Each of the following holds true:
\begin{description}
\item{(i)} If the zero equilibrium of Eq. \eqref{Eq-Linear3} is globally attracting for some $m,$ then the zero equilibrium of Eq. \eqref{Eq-Linear1} is globally attracting.
\item{(ii)} If the zero equilibrium of Eq. \eqref{Eq-Linear1} is globally attracting, and the zeros of $q_m$
are in $\mathbb{D}$, then the zero equilibrium of Eq. \eqref{Eq-Linear3} is globally attracting.
\end{description}
\end{theorem}
\begin{proof}
Both cases follow directly from the relationship between the polynomials $p_0$ and $p_m$ that we established in Lemma \ref{Lem-Pm}. Recall that $p_0$ and $p_m$ are the characteristic polynomials of the Jacobian matrices in systems  \eqref{Eq-Linear1} and \eqref{Eq-Linear3}, respectively.  
\end{proof}
\begin{table}[H]
\caption{This table shows the polynomials $q_m$ defined in Eq. \eqref{Eq-Def-Pm} and guaranteed by Lemma \ref{Lem-Pm}. Also, the zeros of $q_m$ are given in the second column. } \label{Tab-2}
\centering
\begin{tabular}{lclcc}
\hline\hline\\ [-1.5ex]
 $m$ &&$q_m(x)$&& The zeros \\ [0.5ex]
\hline
&&&&\\
 $0$& &$1$&& None\\[2ex]
$1$&& $x+\frac{5}{4}$&&  $-1.250$ \\[2ex]
$2$&& $x^2+\frac{5}{4}x+\frac{19}{16}$&&  $-0.625\pm 0.893i$ \\[2ex]
$3$&& $x^3+\frac{5}{4}x^2+ \frac{19}{16}x + \frac{65}{64}$&&  $-0.103 \pm 0.981i, -1.044$ \\[2ex]
$4$&& $x^4 + \frac{5}{4}x^3 + \frac{19}{16}x^2 + \frac{65}{64}x + \frac{211}{256}$&&  $0.193 \pm 0.901i, -0.818\pm  0.548i$ \\[2ex]
\hline
\end{tabular}
\end{table}
\begin{theorem}\label{Th-MainTheorem}
Consider the polynomial $p_m$ as defined in Eq. \eqref{Eq-Def-Pm}. The zero equilibrium of Eq. \eqref{Eq-Linear1} is globally attracting if and only if there exists $m\in \mathbb{N}$ such that $\|p_m\|_{\ell_1}<2.$
\end{theorem}
\begin{proof}
($\Leftarrow$) Suppose there exists $m\in\mathbb{N}$ such that $\|p_m\|_{\ell_1}<2,$  then the zeros of $p_m$ are in $\mathbb{D},$ and the zero equilibrium of Eq. \eqref{Eq-Linear1} is globally attracting by Part (i) of Theorem \ref{Th-Stability1}.

($\Rightarrow$) Suppose the zero equilibrium of Eq. \eqref{Eq-Linear1} is globally attracting. Based on Eq. \eqref{Eq-Vm}, the stability of the zero equilibrium of Eq. \eqref{Eq-Linear1} means $Spec(J_0)=Spec(J_0^t)\subset \mathbb{D},$ and consequently, $\|V_m\|_1$ converges to zero. Hence, there exists a finite number $m$ such that $\|V_m\|_1<1,$ which means that $\|p_m\|_{\ell_1}<2$ as needed.
\end{proof}

To clarify the computational process of Theorem \ref{Th-MainTheorem} further, we revisit the motivational example $x_{n+1}=a_0x_n+a_1x_{n-1}$ that was considered in \cite{Al-Ca-Ka2025}. But here, we illustrate it in terms of the eigenvalues rather than the coefficients. We write the eigenvalues in terms of the coefficients to obtain $x_{n+1}=(\lambda_1+\lambda_2)x_n-\lambda_1\lambda_2x_{n-1}.$ Then we apply the expansion technique for $m=0,1$ and $2.$ We separate the case of real eigenvalues from the case of non-real eigenvalues and plot Fig. \ref{Fig-Eigenvalues} to illustrate how our expansion process captures the eigenvalues based on their location within $\mathbb{D}$.
%%%%%%%%%%%%%%%%%%%%%%%%%%%%%%%%%%%%%%%%%%%%%%%%%%
%%%%%%%%%%%%%%%%%%%%%%%%%%%%%%%%%%%%%%%%%%%%%%%%%%
\definecolor{MyMaroon}{rgb}{128,0,0}
\definecolor{MyOrange}{rgb}{255,87,51}
\definecolor{ffvvqq}{rgb}{1,0.3333333333333333,0}
\definecolor{qqqqff}{rgb}{0.,0.,1.}
\definecolor{ffqqqq}{rgb}{1.,0.,0.}
\definecolor{cqcqcq}{rgb}{0.7529,0.7529,0.7529}
\definecolor{MyGreen}{rgb}{0,0.50196,0}
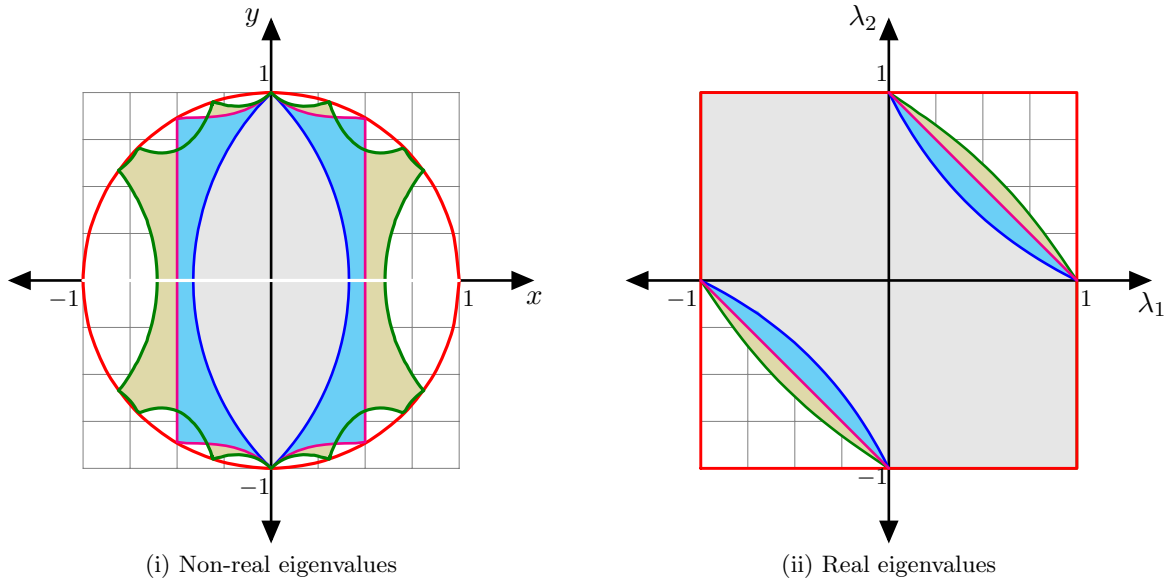
\begin{figure}[htbp]
\centering
\begin{minipage}[t]{0.5\textwidth}
%\raggedright
\begin{center}
\begin{tikzpicture}[line cap=round,line join=round,>=triangle 45,x=1.0cm,y=1.0cm,scale=2.5]
\draw[help lines,step=.25] (-1.0,-1.0) grid (1.0,1.0);
\fill[line width=1pt,color=ffqqqq,fill=olive!30,fill opacity=1.0] (0,1)-- (0.01966,0.982)-- (0.0347,0.97075)-- (0.0508,0.9604)-- (0.0685,0.951)-- (0.087,0.943)-- (0.105,0.9374)-- (0.124,0.933)-- (0.151,0.929)-- (0.171,0.927)-- (0.194,0.928)-- (0.212,0.929)-- (0.229,0.93)-- (0.25,0.935)-- (0.265,0.94)-- (0.287,0.944)-- (0.306,0.95)--(0.315,0.933)-- (0.31,0.95)-- (0.322,0.91)-- (0.331,0.89)-- (0.34,0.87)-- (0.35,0.85)-- (0.359,0.8344)-- (0.3704,0.81447)-- (0.383,0.795)-- (0.396,0.778)-- (0.408,0.763)-- (0.422,0.75)-- (0.436,0.735)-- (0.45,0.724)-- (0.464,0.713)-- (0.48,0.703)-- (0.49786,0.695)-- (0.5177,0.6878)-- (0.54,0.683)-- (0.557,0.68)-- (0.579,0.6785)-- (0.60,0.679)-- (0.6213,0.6818)-- (0.644,0.686)-- (0.672,0.694)-- (0.6987,0.7037)--(0.6987,0.7037) -- (0.713,0.694) -- (0.725,0.673) -- (0.737,0.6556) -- (0.748,0.642) -- (0.761,0.62867) -- (0.774,0.616) -- (0.797,0.5965) -- (0.812,0.586) -- (0.794,0.566) -- (0.778,0.544) -- (0.76678,0.5265) -- (0.757,0.512) -- (0.74794,0.497) -- (0.731,0.468) -- (0.722,0.45) -- (0.71,0.432) -- (0.696,0.402) -- (0.675,0.354) -- (0.66,0.321) -- (0.655,0.30) -- (0.6457,0.2724) -- (0.6404,0.254) -- (0.633,0.228) -- (0.626898,0.1999) -- (0.6233,0.18) -- (0.6175,0.1498) -- (0.613,0.119) -- (0.611,0.0988) -- (0.6087,0.075) -- (0.607,0.049) -- (0.606,0.0223) -- (0.6058,0)--(0.606,-0.0223)--(0.607,-0.049)-- (0.6087,-0.075) -- (0.611,-0.0988) -- (0.613,-0.119) -- (0.6175,-0.1498) -- (0.6233,-0.18) -- (0.626898,-0.1999) -- (0.633,-0.228) -- (0.6404,-0.254) -- (0.6457,-0.2724)-- (0.655,-0.30) -- (0.66,-0.321)-- (0.675,-0.354)-- (0.696,-0.402) -- (0.71,-0.432) -- (0.722,-0.45)-- (0.731,-0.468) -- (0.74794,-0.497)-- (0.757,-0.512) -- (0.76678,-0.5265) -- (0.778,-0.544)-- (0.794,-0.566) -- (0.812,-0.586) -- (0.797,-0.5965) -- (0.774,-0.616)-- (0.761,-0.62867)-- (0.748,-0.642)-- (0.737,-0.6556)--(0.725,-0.673) -- (0.713,-0.694)--(0.6987,-0.7037)-- (0.6987,-0.7037)-- (0.672,-0.694)-- (0.644,-0.686)--(0.6213,-0.6818)-- (0.60,-0.679)-- (0.579,-0.6785)-- (0.557,-0.68)-- (0.54,-0.683)-- (0.5177,-0.6878)--(0.49786,-0.695)-- (0.48,-0.703)-- (0.464,-0.713)-- (0.45,-0.724)-- (0.436,-0.735)-- (0.422,-0.75)--(0.408,-0.763)-- (0.396,-0.778)-- (0.383,-0.795)-- (0.3704,-0.81447)-- (0.359,-0.8344)-- (0.35,-0.85)--(0.34,-0.87)-- (0.331,-0.89)-- (0.322,-0.91)-- (0.31,-0.95)--(0.315,-0.933)-- (0.306,-0.95)--(0.287,-0.944)-- (0.265,-0.94)-- (0.25,-0.935)-- (0.229,-0.93)-- (0.212,-0.93)-- (0.194,-0.928)--(0.171,-0.927)-- (0.151,-0.929)-- (0.124,-0.933)-- (0.105,-0.9374)-- (0.087,-0.943)-- (0.0685,-0.951)-- (0.0508,-0.9604)-- (0.0347,-0.97075)-- (0.01966,-0.982)--(0,-1)--
(-0.01966,-0.982)-- (-0.0347,-0.97075)-- (-0.0508,-0.9604)-- (-0.0685,-0.951)-- (-0.087,-0.943)-- (-0.105,-0.9374)-- (-0.124,-0.933)-- (-0.151,-0.929)-- (-0.171,-0.927)-- (-0.194,-0.928)-- (-0.212,-0.929)-- (-0.229,-0.93)-- (-0.25,-0.935)-- (-0.265,-0.94)-- (-0.287,-0.944)-- (-0.306,-0.95)--(-0.315,-0.933)-- (-0.31,-0.95)-- (-0.322,-0.91)-- (-0.331,-0.89)-- (-0.34,-0.87)-- (-0.35,-0.85)-- (-0.359,-0.8344)-- (-0.3704,-0.81447)-- (-0.383,-0.795)-- (-0.396,-0.778)-- (-0.408,-0.763)-- (-0.422,-0.75)-- (-0.436,-0.735)-- (-0.45,-0.724)-- (-0.464,-0.713)-- (-0.48,-0.703)-- (-0.49786,-0.695)-- (-0.5177,-0.6878)-- (-0.54,-0.683)-- (-0.557,-0.68)-- (-0.579,-0.6785)-- (-0.60,-0.679)-- (-0.6213,-0.6818)-- (-0.644,-0.686)-- (-0.672,-0.694)-- (-0.6987,-0.7037)--(-0.6987,-0.7037) -- (-0.713,-0.694) -- (-0.725,-0.673) -- (-0.737,-0.6556) -- (-0.748,-0.642) -- (-0.761,-0.62867) -- (-0.774,-0.616) -- (-0.797,-0.5965) -- (-0.812,-0.586) -- (-0.794,-0.566) -- (-0.778,-0.544) -- (-0.76678,-0.5265) -- (-0.757,-0.512) -- (-0.74794,-0.497) -- (-0.731,-0.468) -- (-0.722,-0.45) -- (-0.71,-0.432) -- (-0.696,-0.402) -- (-0.675,-0.354) -- (-0.66,-0.321) -- (-0.655,-0.30) -- (-0.6457,-0.2724) -- (-0.6404,-0.254) -- (-0.633,-0.228) -- (-0.626898,-0.1999) -- (-0.6233,-0.18) -- (-0.6175,-0.1498) -- (-0.613,-0.119) -- (-0.611,-0.0988) -- (-0.6087,-0.075) -- (-0.607,-0.049) -- (-0.606,-0.0223) -- (-0.6058,0)--(-0.606,0.0223)--(-0.607,0.049)-- (-0.6087,0.075) -- (-0.611,0.0988) -- (-0.613,0.119) -- (-0.6175,0.1498) -- (-0.6233,0.18) -- (-0.626898,0.1999) -- (-0.633,0.228) -- (-0.6404,0.254) -- (-0.6457,0.2724)-- (-0.655,0.30) -- (-0.66,0.321)-- (-0.675,0.354)-- (-0.696,0.402) -- (-0.71,0.432) -- (-0.722,0.45)-- (-0.731,0.468) -- (-0.74794,0.497)-- (-0.757,0.512) -- (-0.76678,0.5265) -- (-0.778,0.544)-- (-0.794,0.566) -- (-0.812,0.586) -- (-0.797,0.5965) -- (-0.774,0.616)-- (-0.761,0.62867)-- (-0.748,0.642)-- (-0.737,0.6556)--(-0.725,0.673) -- (-0.713,0.694)--(-0.6987,0.7037)-- (-0.6987,0.7037)-- (-0.672,0.694)-- (-0.644,0.686)--(-0.6213,0.6818)-- (-0.60,0.679)-- (-0.579,0.6785)-- (-0.557,0.68)-- (-0.54,0.683)-- (-0.5177,0.6878)--(-0.49786,0.695)-- (-0.48,0.703)-- (-0.464,0.713)-- (-0.45,0.724)-- (-0.436,0.735)-- (-0.422,0.75)--(-0.408,0.763)-- (-0.396,0.778)-- (-0.383,0.795)-- (-0.3704,0.81447)-- (-0.359,0.8344)-- (-0.35,0.85)--(-0.34,0.87)-- (-0.331,0.89)-- (-0.322,0.91)-- (-0.31,0.95)--(-0.315,0.933)-- (-0.306,0.95)--(-0.287,0.944)-- (-0.265,0.94)-- (-0.25,0.935)-- (-0.229,0.93)-- (-0.212,0.93)-- (-0.194,0.928)--(-0.171,0.927)-- (-0.151,0.929)-- (-0.124,0.933)-- (-0.105,0.9374)-- (-0.087,0.943)-- (-0.0685,0.951)-- (-0.0508,0.9604)-- (-0.0347,0.97075)-- (-0.01966,0.982)--(0,1)--cycle;
\fill[fill=cyan!50,fill opacity=1.0] (0.0,1) -- (0.0498,0.957) -- (0.0939,0.9288) -- (0.1402,0.907) -- (0.197,0.889) -- (0.2395,0.8797) -- (0.28968,0.873) -- (0.337,0.86899) -- (0.3898,0.86689) -- (0.45036,0.866) -- (0.5,0.86) -- (0.5,-0.86) -- (0.45036,-0.866) -- (0.3898,-0.86689) -- (0.337,-0.86899) -- (0.28968,-0.873) -- (0.2395,-0.8797) -- (0.197,-0.889) -- (0.1402,-0.907) -- (0.0939,-0.9288) -- (0.0498,-0.957) -- (0,-1)--(-0.0498,-0.957) -- (-0.0939,-0.9288) -- (-0.1402,-0.907) -- (-0.197,-0.889) -- (-0.2395,-0.8797) -- (-0.28968,-0.873) -- (-0.337,-0.86899) -- (-0.3898,-0.86689) -- (-0.45036,-0.866) -- (-0.5,-0.86) -- (-0.5,0.86) -- (-0.45036,0.866) -- (-0.3898,0.86689) -- (-0.337,0.86899) -- (-0.28968,0.873) -- (-0.2395,0.8797) -- (-0.197,0.889) -- (-0.1402,0.907) -- (-0.0939,0.9288) -- (-0.0498,0.957) -- (0,1)--cycle;
\fill[fill=gray!20] (0.00151545243597873,0.9984822498342265) -- (0.04130373618398374,0.9569151065907501) -- (0.06981865394145502,0.9249259698502412) -- (0.11497022645483136,0.8699663138170168) -- (0.14260652780950278,0.8333368566128772) -- (0.17251052290282334,0.7907079539973063) -- (0.19762739241962984,0.7521227441759761) -- (0.2239804122944468,0.7084292082294295) -- (0.24458080796106352,0.6715791945579095) -- (0.2668315389724596,0.6285999122149596) -- (0.2865409660827387,0.5872072367868147) -- (0.30398768114908015,0.5473720185079898) -- (0.32266180629877583,0.5005654237708179) -- (0.33824332979647315,0.45727977014456145) -- (0.35166220405356363,0.4159438578450644) -- (0.36588960135172655,0.3665318512363239) -- (0.3763467837659375,0.32506850801844467) -- (0.38629337194048874,0.2796259777848409) -- (0.3949783627832052,0.23245509598331238) -- (0.40000000143736614,0.20000000721313832) -- (0.40640628069953966,0.14839600544963943) -- (0.41000693557315715,0.10899745386068932) -- (0.4128710695480359,0.061606349954752906) -- (0.41421356354422356,-0.0000024297589152750923) -- (0.41353862876802294,-0.043686851158145495) -- (0.41108570124170235,-0.09400606051699104) -- (0.40359782178662285,-0.1729542040517701) -- (0.39516890717718806,-0.23130870013878516) -- (0.3864859279948552,-0.27866962790110256) -- (0.37502957091225364,-0.33059594028710176) -- (0.3611457999466277,-0.3837735094753413) -- (0.34595477193091206,-0.4340573159914135) -- (0.32334090452698205,-0.49876732959567305) -- (0.3078930238125526,-0.5379738264304494) -- (0.2897357237686963,-0.5801566746825506) -- (0.2697278008254521,-0.6227289182853352) -- (0.2466208579635466,-0.6677847221671896) -- (0.2271480416780325,-0.7029279309640825) -- (0.1973681192388362,-0.7525354392320721) -- (0.17141835895872692,-0.7923250755041773) -- (0.142904502779108,-0.8329281481966182) -- (0.11922282484862656,-0.864488441879371) -- (0.08286803178373536,-0.9096135567319031) -- (0.055979763999803894,-0.9406948189670412) -- (0.001514698231981981,-0.9984830080154887) -- (-0.032733437016226385,-0.9661581915239399) -- (-0.06670470877910015,-0.9285155198883904) -- (-0.09415910780083517,-0.8959999177447979) -- (-0.12225518664636503,-0.86054825494449) -- (-0.15015604715405007,-0.822885819433363) -- (-0.17266056085056256,-0.7904854302670945) -- (-0.19537065585508617,-0.7557043037494353) -- (-0.22159202978985088,-0.7125397619343328) -- (-0.24582111602162435,-0.6692755354023079) -- (-0.2647667006758452,-0.6327441834625082) -- (-0.28317577756386403,-0.5945249569072795) -- (-0.3022427375321743,-0.5515105217176488) -- (-0.3340914189950677,-0.46925481772616384) -- (-0.3510950813466851,-0.4177823387144838) -- (-0.36396177376037625,-0.3736419101539599) -- (-0.37874306850119727,-0.3147499814650379) -- (-0.388820615247759,-0.2667907417254215) -- (-0.3999999955862375,-0.2000000206702649) -- (-0.40633152985366505,-0.14910272776778255) -- (-0.41206018454040205,-0.0780130793003668) -- (-0.4142135646907452,-0.000002433580518049152) -- (-0.41237026182415065,0.07218199707105666) -- (-0.407756173087595,0.13499093211631502) -- (-0.3999999950633795,0.20000002224753766) -- (-0.3879438860858765,0.27131490469114994) -- (-0.36362053565286256,0.374885358969402) -- (-0.3413168794771855,0.44818414903357207) -- (-0.3123232528759561,0.5270746489649878) -- (-0.27684340879187963,0.6080056874885782) -- (-0.23429050421189612,0.6903093149280422) -- (-0.1896736860812534,0.7646414357905467) -- (-0.14146753164963788,0.834896323127536) -- (-0.07520933633064661,0.9186538443228096) -- cycle;
\draw[-triangle 45, line width=1.0pt,scale=1] (0,0) -- (1.4,0) node[below] {$x$};
\draw[line width=1.0pt,-triangle 45] (0,0) -- (-1.4,0);
\draw[-triangle 45, line width=1.0pt,scale=1] (0,0) -- (0,1.4) node[left] {$y$};
\draw[line width=1.0pt,-triangle 45] (0,0) -- (0.0,-1.4);
%\draw [samples=50,domain=-2.0:2.0,xshift=0.0cm,yshift=0.0cm,line width=1.0pt,color=ffqqqq,fill=ffqqqq,pattern=dots,pattern color=ffqqqq] plot ({\x},{0.25*\x*\x});
\draw[line width=1.2pt,domain=0:360,smooth,variable=\t,color=red] plot ({cos(\t)},{sin(\t)});
\draw[line width=1.0pt,color=blue,smooth] (0.00151545243597873,0.9984822498342265) -- (0.04130373618398374,0.9569151065907501) -- (0.06981865394145502,0.9249259698502412) -- (0.11497022645483136,0.8699663138170168) -- (0.14260652780950278,0.8333368566128772) -- (0.17251052290282334,0.7907079539973063) -- (0.19762739241962984,0.7521227441759761) -- (0.2239804122944468,0.7084292082294295) -- (0.24458080796106352,0.6715791945579095) -- (0.2668315389724596,0.6285999122149596) -- (0.2865409660827387,0.5872072367868147) -- (0.30398768114908015,0.5473720185079898) -- (0.32266180629877583,0.5005654237708179) -- (0.33824332979647315,0.45727977014456145) -- (0.35166220405356363,0.4159438578450644) -- (0.36588960135172655,0.3665318512363239) -- (0.3763467837659375,0.32506850801844467) -- (0.38629337194048874,0.2796259777848409) -- (0.3949783627832052,0.23245509598331238) -- (0.40000000143736614,0.20000000721313832) -- (0.40640628069953966,0.14839600544963943) -- (0.41000693557315715,0.10899745386068932) -- (0.4128710695480359,0.061606349954752906) -- (0.41421356354422356,-0.0000024297589152750923) -- (0.41353862876802294,-0.043686851158145495) -- (0.41108570124170235,-0.09400606051699104) -- (0.40359782178662285,-0.1729542040517701) -- (0.39516890717718806,-0.23130870013878516) -- (0.3864859279948552,-0.27866962790110256) -- (0.37502957091225364,-0.33059594028710176) -- (0.3611457999466277,-0.3837735094753413) -- (0.34595477193091206,-0.4340573159914135) -- (0.32334090452698205,-0.49876732959567305) -- (0.3078930238125526,-0.5379738264304494) -- (0.2897357237686963,-0.5801566746825506) -- (0.2697278008254521,-0.6227289182853352) -- (0.2466208579635466,-0.6677847221671896) -- (0.2271480416780325,-0.7029279309640825) -- (0.1973681192388362,-0.7525354392320721) -- (0.17141835895872692,-0.7923250755041773) -- (0.142904502779108,-0.8329281481966182) -- (0.11922282484862656,-0.864488441879371) -- (0.08286803178373536,-0.9096135567319031) -- (0.055979763999803894,-0.9406948189670412) -- (0.001514698231981981,-0.9984830080154887) -- (-0.032733437016226385,-0.9661581915239399) -- (-0.06670470877910015,-0.9285155198883904) -- (-0.09415910780083517,-0.8959999177447979) -- (-0.12225518664636503,-0.86054825494449) -- (-0.15015604715405007,-0.822885819433363) -- (-0.17266056085056256,-0.7904854302670945) -- (-0.19537065585508617,-0.7557043037494353) -- (-0.22159202978985088,-0.7125397619343328) -- (-0.24582111602162435,-0.6692755354023079) -- (-0.2647667006758452,-0.6327441834625082) -- (-0.28317577756386403,-0.5945249569072795) -- (-0.3022427375321743,-0.5515105217176488) -- (-0.3340914189950677,-0.46925481772616384) -- (-0.3510950813466851,-0.4177823387144838) -- (-0.36396177376037625,-0.3736419101539599) -- (-0.37874306850119727,-0.3147499814650379) -- (-0.388820615247759,-0.2667907417254215) -- (-0.3999999955862375,-0.2000000206702649) -- (-0.40633152985366505,-0.14910272776778255) -- (-0.41206018454040205,-0.0780130793003668) -- (-0.4142135646907452,-0.000002433580518049152) -- (-0.41237026182415065,0.07218199707105666) -- (-0.407756173087595,0.13499093211631502) -- (-0.3999999950633795,0.20000002224753766) -- (-0.3879438860858765,0.27131490469114994) -- (-0.36362053565286256,0.374885358969402) -- (-0.3413168794771855,0.44818414903357207) -- (-0.3123232528759561,0.5270746489649878) -- (-0.27684340879187963,0.6080056874885782) -- (-0.23429050421189612,0.6903093149280422) -- (-0.1896736860812534,0.7646414357905467) -- (-0.14146753164963788,0.834896323127536) -- (-0.07520933633064661,0.9186538443228096) -- cycle;
\draw[line width=1pt,color=magenta,smooth] (0.0,1) -- (0.0498,0.957) -- (0.0939,0.9288) -- (0.1402,0.907) -- (0.197,0.889) -- (0.2395,0.8797) -- (0.28968,0.873) -- (0.337,0.86899) -- (0.3898,0.86689) -- (0.45036,0.866) -- (0.5,0.86) -- (0.5,-0.86) -- (0.45036,-0.866) -- (0.3898,-0.86689) -- (0.337,-0.86899) -- (0.28968,-0.873) -- (0.2395,-0.8797) -- (0.197,-0.889) -- (0.1402,-0.907) -- (0.0939,-0.9288) -- (0.0498,-0.957) -- (0,-1)--(-0.0498,-0.957) -- (-0.0939,-0.9288) -- (-0.1402,-0.907) -- (-0.197,-0.889) -- (-0.2395,-0.8797) -- (-0.28968,-0.873) -- (-0.337,-0.86899) -- (-0.3898,-0.86689) -- (-0.45036,-0.866) -- (-0.5,-0.86) -- (-0.5,0.86) -- (-0.45036,0.866) -- (-0.3898,0.86689) -- (-0.337,0.86899) -- (-0.28968,0.873) -- (-0.2395,0.8797) -- (-0.197,0.889) -- (-0.1402,0.907) -- (-0.0939,0.9288) -- (-0.0498,0.957) -- (0,1);
\draw[line width=1.2pt,color=MyGreen,smooth] (0,1)-- (0.01966,0.982)-- (0.0347,0.97075)-- (0.0508,0.9604)-- (0.0685,0.951)-- (0.087,0.943)-- (0.105,0.9374)-- (0.124,0.933)-- (0.151,0.929)-- (0.171,0.927)-- (0.194,0.928)-- (0.212,0.929)-- (0.229,0.93)-- (0.25,0.935)-- (0.265,0.94)-- (0.287,0.944)-- (0.306,0.95)--(0.315,0.933)-- (0.31,0.95)-- (0.322,0.91)-- (0.331,0.89)-- (0.34,0.87)-- (0.35,0.85)-- (0.359,0.8344)-- (0.3704,0.81447)-- (0.383,0.795)-- (0.396,0.778)-- (0.408,0.763)-- (0.422,0.75)-- (0.436,0.735)-- (0.45,0.724)-- (0.464,0.713)-- (0.48,0.703)-- (0.49786,0.695)-- (0.5177,0.6878)-- (0.54,0.683)-- (0.557,0.68)-- (0.579,0.6785)-- (0.60,0.679)-- (0.6213,0.6818)-- (0.644,0.686)-- (0.672,0.694)-- (0.6987,0.7037)--(0.6987,0.7037) -- (0.713,0.694) -- (0.725,0.673) -- (0.737,0.6556) -- (0.748,0.642) -- (0.761,0.62867) -- (0.774,0.616) -- (0.797,0.5965) -- (0.812,0.586) -- (0.794,0.566) -- (0.778,0.544) -- (0.76678,0.5265) -- (0.757,0.512) -- (0.74794,0.497) -- (0.731,0.468) -- (0.722,0.45) -- (0.71,0.432) -- (0.696,0.402) -- (0.675,0.354) -- (0.66,0.321) -- (0.655,0.30) -- (0.6457,0.2724) -- (0.6404,0.254) -- (0.633,0.228) -- (0.626898,0.1999) -- (0.6233,0.18) -- (0.6175,0.1498) -- (0.613,0.119) -- (0.611,0.0988) -- (0.6087,0.075) -- (0.607,0.049) -- (0.606,0.0223) -- (0.6058,0)--(0.606,-0.0223)--(0.607,-0.049)-- (0.6087,-0.075) -- (0.611,-0.0988) -- (0.613,-0.119) -- (0.6175,-0.1498) -- (0.6233,-0.18) -- (0.626898,-0.1999) -- (0.633,-0.228) -- (0.6404,-0.254) -- (0.6457,-0.2724)-- (0.655,-0.30) -- (0.66,-0.321)-- (0.675,-0.354)-- (0.696,-0.402) -- (0.71,-0.432) -- (0.722,-0.45)-- (0.731,-0.468) -- (0.74794,-0.497)-- (0.757,-0.512) -- (0.76678,-0.5265) -- (0.778,-0.544)-- (0.794,-0.566) -- (0.812,-0.586) -- (0.797,-0.5965) -- (0.774,-0.616)-- (0.761,-0.62867)-- (0.748,-0.642)-- (0.737,-0.6556)--(0.725,-0.673) -- (0.713,-0.694)--(0.6987,-0.7037)-- (0.6987,-0.7037)-- (0.672,-0.694)-- (0.644,-0.686)--(0.6213,-0.6818)-- (0.60,-0.679)-- (0.579,-0.6785)-- (0.557,-0.68)-- (0.54,-0.683)-- (0.5177,-0.6878)--(0.49786,-0.695)-- (0.48,-0.703)-- (0.464,-0.713)-- (0.45,-0.724)-- (0.436,-0.735)-- (0.422,-0.75)--(0.408,-0.763)-- (0.396,-0.778)-- (0.383,-0.795)-- (0.3704,-0.81447)-- (0.359,-0.8344)-- (0.35,-0.85)--(0.34,-0.87)-- (0.331,-0.89)-- (0.322,-0.91)-- (0.31,-0.95)--(0.315,-0.933)-- (0.306,-0.95)--(0.287,-0.944)-- (0.265,-0.94)-- (0.25,-0.935)-- (0.229,-0.93)-- (0.212,-0.93)-- (0.194,-0.928)--(0.171,-0.927)-- (0.151,-0.929)-- (0.124,-0.933)-- (0.105,-0.9374)-- (0.087,-0.943)-- (0.0685,-0.951)-- (0.0508,-0.9604)-- (0.0347,-0.97075)-- (0.01966,-0.982)--(0,-1)--
(-0.01966,-0.982)-- (-0.0347,-0.97075)-- (-0.0508,-0.9604)-- (-0.0685,-0.951)-- (-0.087,-0.943)-- (-0.105,-0.9374)-- (-0.124,-0.933)-- (-0.151,-0.929)-- (-0.171,-0.927)-- (-0.194,-0.928)-- (-0.212,-0.929)-- (-0.229,-0.93)-- (-0.25,-0.935)-- (-0.265,-0.94)-- (-0.287,-0.944)-- (-0.306,-0.95)--(-0.315,-0.933)-- (-0.31,-0.95)-- (-0.322,-0.91)-- (-0.331,-0.89)-- (-0.34,-0.87)-- (-0.35,-0.85)-- (-0.359,-0.8344)-- (-0.3704,-0.81447)-- (-0.383,-0.795)-- (-0.396,-0.778)-- (-0.408,-0.763)-- (-0.422,-0.75)-- (-0.436,-0.735)-- (-0.45,-0.724)-- (-0.464,-0.713)-- (-0.48,-0.703)-- (-0.49786,-0.695)-- (-0.5177,-0.6878)-- (-0.54,-0.683)-- (-0.557,-0.68)-- (-0.579,-0.6785)-- (-0.60,-0.679)-- (-0.6213,-0.6818)-- (-0.644,-0.686)-- (-0.672,-0.694)-- (-0.6987,-0.7037)--(-0.6987,-0.7037) -- (-0.713,-0.694) -- (-0.725,-0.673) -- (-0.737,-0.6556) -- (-0.748,-0.642) -- (-0.761,-0.62867) -- (-0.774,-0.616) -- (-0.797,-0.5965) -- (-0.812,-0.586) -- (-0.794,-0.566) -- (-0.778,-0.544) -- (-0.76678,-0.5265) -- (-0.757,-0.512) -- (-0.74794,-0.497) -- (-0.731,-0.468) -- (-0.722,-0.45) -- (-0.71,-0.432) -- (-0.696,-0.402) -- (-0.675,-0.354) -- (-0.66,-0.321) -- (-0.655,-0.30) -- (-0.6457,-0.2724) -- (-0.6404,-0.254) -- (-0.633,-0.228) -- (-0.626898,-0.1999) -- (-0.6233,-0.18) -- (-0.6175,-0.1498) -- (-0.613,-0.119) -- (-0.611,-0.0988) -- (-0.6087,-0.075) -- (-0.607,-0.049) -- (-0.606,-0.0223) -- (-0.6058,0)--(-0.606,0.0223)--(-0.607,0.049)-- (-0.6087,0.075) -- (-0.611,0.0988) -- (-0.613,0.119) -- (-0.6175,0.1498) -- (-0.6233,0.18) -- (-0.626898,0.1999) -- (-0.633,0.228) -- (-0.6404,0.254) -- (-0.6457,0.2724)-- (-0.655,0.30) -- (-0.66,0.321)-- (-0.675,0.354)-- (-0.696,0.402) -- (-0.71,0.432) -- (-0.722,0.45)-- (-0.731,0.468) -- (-0.74794,0.497)-- (-0.757,0.512) -- (-0.76678,0.5265) -- (-0.778,0.544)-- (-0.794,0.566) -- (-0.812,0.586) -- (-0.797,0.5965) -- (-0.774,0.616)-- (-0.761,0.62867)-- (-0.748,0.642)-- (-0.737,0.6556)--(-0.725,0.673) -- (-0.713,0.694)--(-0.6987,0.7037)-- (-0.6987,0.7037)-- (-0.672,0.694)-- (-0.644,0.686)--(-0.6213,0.6818)-- (-0.60,0.679)-- (-0.579,0.6785)-- (-0.557,0.68)-- (-0.54,0.683)-- (-0.5177,0.6878)--(-0.49786,0.695)-- (-0.48,0.703)-- (-0.464,0.713)-- (-0.45,0.724)-- (-0.436,0.735)-- (-0.422,0.75)--(-0.408,0.763)-- (-0.396,0.778)-- (-0.383,0.795)-- (-0.3704,0.81447)-- (-0.359,0.8344)-- (-0.35,0.85)--(-0.34,0.87)-- (-0.331,0.89)-- (-0.322,0.91)-- (-0.31,0.95)--(-0.315,0.933)-- (-0.306,0.95)--(-0.287,0.944)-- (-0.265,0.94)-- (-0.25,0.935)-- (-0.229,0.93)-- (-0.212,0.93)-- (-0.194,0.928)--(-0.171,0.927)-- (-0.151,0.929)-- (-0.124,0.933)-- (-0.105,0.9374)-- (-0.087,0.943)-- (-0.0685,0.951)-- (-0.0508,0.9604)-- (-0.0347,0.97075)-- (-0.01966,0.982)--(0,1)--cycle;
\draw[line width=1.2pt,color=white,smooth] (-1,0)--(1,0);
\draw[scale=1] (1.05,0) node[below] {\footnotesize $1 $};
\draw[scale=1] (0.05,1.1) node[left] {\footnotesize $1 $};
\draw[scale=1] (-1.1,0) node[below] {\footnotesize $-1 $};
\draw[scale=1] (0.05,-1.1) node[left] {\footnotesize $-1 $};
\draw[scale=1] (0.0,-1.4) node[below,rotate=0] {\footnotesize (i) Non-real eigenvalues};
\end{tikzpicture}
\end{center}
\end{minipage}%
%%%%%%%%%%%%%%%%%%%%%%%%%%%%%%%%%%%%%%%%%%%%%%%%%%%%%%%%%%%%%%%%%%%%%%%%%%%%%%%%%%%%%%%%%%%%
\begin{minipage}[t]{0.5\textwidth}
%\raggedright
\begin{center}
\begin{tikzpicture}[line cap=round,line join=round,>=triangle 45,x=1.0cm,y=1.0cm,scale=2.5]
\draw[help lines,step=.25] (-1.0,-1.0) grid (1.0,1.0);
\fill[line width=1pt,color=ffqqqq,fill=olive!30,fill opacity=1.0] (-1,0) -- (-1,1) -- (0,1) -- (0.0409,0.973) -- (0.0785,0.9474) -- (0.117,0.922) -- (0.159,0.89) -- (0.198,0.865) -- (0.2377,0.836) -- (0.271,0.811) -- (0.323,0.771) -- (0.374,0.73) -- (0.416,0.695) -- (0.451,0.664) -- (0.48,0.6364) -- (0.5178,0.602) -- (0.553,0.568) -- (0.585,0.535) -- (0.6264,0.492) -- (0.657,0.459) -- (0.686,0.4253) -- (0.714,0.393) -- (0.745,0.355) -- (0.78,0.311) -- (0.8188,0.261) -- (0.8535,0.213) -- (0.887,0.166) -- (0.9195,0.1196) -- (0.944,0.083) -- (0.9698,0.045) -- (1,0.0) -- (1,-1) -- (-0.00,-0.99999) -- (-0.0634,-0.9576) -- (-0.11419,-0.923) -- (-0.15327,-0.896) -- (-0.2075,-0.858) -- (-0.26476,-0.816) -- (-0.307,-0.783) -- (-0.3527,-0.747) -- (-0.3945,-0.713) -- (-0.441,-0.673) -- (-0.487,-0.63) -- (-0.54,-0.58) -- (-0.587,-0.534) -- (-0.616,-0.503) -- (-0.653,-0.463) -- (-0.687,-0.424) -- (-0.715,-0.391) -- (-0.7463,-0.354) -- (-0.778,-0.314) -- (-0.814,-0.267) -- (-0.845,-0.225) -- (-0.879,-0.178) -- (-0.909,-0.135) -- (-0.937,-0.094) -- (-1,0) -- cycle;
\fill[fill=gray!20] (0,1) -- (0.0358,0.931) -- (0.0627,0.88) -- (0.0896,0.8355) -- (0.12455,0.778) -- (0.1514,0.737) -- (0.181,0.693) -- (0.221,0.6375) -- (0.262,0.585) -- (0.2912,0.5489) -- (0.326,0.508) -- (0.3557,0.475) -- (0.404,0.424) -- (0.439,0.389788) -- (0.4767,0.354) -- (0.514,0.32) -- (0.549,0.2909) -- (0.584,0.26244) -- (0.627,0.229) -- (0.662,0.203) -- (0.7025,0.1747) -- (0.748,0.144) -- (0.79659,0.113) -- (0.8369,0.08878) -- (0.893,0.056) -- (0.934,0.034) -- (1,0) -- (1,-1) -- (0,-1) -- (-0.02867,-0.944) -- (-0.058,-0.8899) -- (-0.1066,-0.807) -- (-0.139,-0.756) -- (-0.168,-0.712) -- (-0.2007,-0.66567) -- (-0.233,-0.622) -- (-0.2625,-0.584) -- (-0.292,-0.54785) -- (-0.3378,-0.495) -- (-0.373,-0.4569) -- (-0.429,-0.399) -- (-0.464,-0.3659) -- (-0.502,-0.33) -- (-0.55,-0.29) -- (-0.585,-0.2617) -- (-0.623,-0.2325) -- (-0.66,-0.2045) -- (-0.695,-0.163) -- (-0.738,-0.1505) -- (-0.78136,-0.1227) -- (-0.824,-0.096) -- (-0.87,-0.069) -- (-0.918,-0.0425) -- (-0.95878,-0.021) -- (-1,0) -- (-1,1) -- cycle;
\fill[line width=1pt,color=ffqqqq,fill=cyan!50,fill opacity=1.0] (0,-1) -- (-0.02867,-0.944) -- (-0.058,-0.8899) -- (-0.1066,-0.807) -- (-0.139,-0.756) -- (-0.168,-0.712) -- (-0.2007,-0.66567) -- (-0.233,-0.622) -- (-0.2625,-0.584) -- (-0.292,-0.54785) -- (-0.3378,-0.495) -- (-0.373,-0.4569) -- (-0.429,-0.399) -- (-0.464,-0.3659) -- (-0.502,-0.33) -- (-0.55,-0.29) -- (-0.585,-0.2617) -- (-0.623,-0.2325) -- (-0.66,-0.2045) -- (-0.695,-0.163) -- (-0.738,-0.1505) -- (-0.78136,-0.1227) -- (-0.824,-0.096) -- (-0.87,-0.069) -- (-0.918,-0.0425) -- (-0.95878,-0.021) -- (-1,0) -- (0,-1)--cycle;
\fill[line width=1pt,color=ffqqqq,fill=cyan!50,fill opacity=1.0] (0,1) -- (0.0358,0.931) -- (0.0627,0.88) -- (0.0896,0.8355) -- (0.12455,0.778) -- (0.1514,0.737) -- (0.181,0.693) -- (0.221,0.6375) -- (0.262,0.585) -- (0.2912,0.5489) -- (0.326,0.508) -- (0.3557,0.475) -- (0.404,0.424) -- (0.439,0.389788) -- (0.4767,0.354) -- (0.514,0.32) -- (0.549,0.2909) -- (0.584,0.26244) -- (0.627,0.229) -- (0.662,0.203) -- (0.7025,0.1747) -- (0.748,0.144) -- (0.79659,0.113) -- (0.8369,0.08878) -- (0.893,0.056) -- (0.934,0.034) -- (1,0)--(0,1)--cycle;
\draw[-triangle 45, line width=1.0pt,scale=1] (0,0) -- (1.4,0) node[below] {$\lambda_1$};
\draw[line width=1.0pt,-triangle 45] (0,0) -- (-1.4,0);
\draw[-triangle 45, line width=1.0pt,scale=1] (0,0) -- (0,1.4) node[left] {$\lambda_2$};
\draw[line width=1.0pt,-triangle 45] (0,0) -- (0.0,-1.4);
\draw[line width=1.0pt,color=blue,smooth] (0,1) -- (0.0358,0.931) -- (0.0627,0.88) -- (0.0896,0.8355) -- (0.12455,0.778) -- (0.1514,0.737) -- (0.181,0.693) -- (0.221,0.6375) -- (0.262,0.585) -- (0.2912,0.5489) -- (0.326,0.508) -- (0.3557,0.475) -- (0.404,0.424) -- (0.439,0.389788) -- (0.4767,0.354) -- (0.514,0.32) -- (0.549,0.2909) -- (0.584,0.26244) -- (0.627,0.229) -- (0.662,0.203) -- (0.7025,0.1747) -- (0.748,0.144) -- (0.79659,0.113) -- (0.8369,0.08878) -- (0.893,0.056) -- (0.934,0.034) -- (1,0);
\draw[line width=1.0pt,color=blue,smooth] (0,-1) -- (-0.02867,-0.944) -- (-0.058,-0.8899) -- (-0.1066,-0.807) -- (-0.139,-0.756) -- (-0.168,-0.712) -- (-0.2007,-0.66567) -- (-0.233,-0.622) -- (-0.2625,-0.584) -- (-0.292,-0.54785) -- (-0.3378,-0.495) -- (-0.373,-0.4569) -- (-0.429,-0.399) -- (-0.464,-0.3659) -- (-0.502,-0.33) -- (-0.55,-0.29) -- (-0.585,-0.2617) -- (-0.623,-0.2325) -- (-0.66,-0.2045) -- (-0.695,-0.178) -- (-0.738,-0.1505) -- (-0.78136,-0.1227) -- (-0.824,-0.096) -- (-0.87,-0.069) -- (-0.918,-0.0425) -- (-0.95878,-0.021) -- (-1,0);
\draw[line width=1.0pt,color=MyGreen,smooth] (-1,0) -- (-1,1) -- (0,1) -- (0.0409,0.973) -- (0.0785,0.9474) -- (0.117,0.922) -- (0.159,0.89) -- (0.198,0.865) -- (0.2377,0.836) -- (0.271,0.811) -- (0.323,0.771) -- (0.374,0.73) -- (0.416,0.695) -- (0.451,0.664) -- (0.48,0.6364) -- (0.5178,0.602) -- (0.553,0.568) -- (0.585,0.535) -- (0.6264,0.492) -- (0.657,0.459) -- (0.686,0.4253) -- (0.714,0.393) -- (0.745,0.355) -- (0.78,0.311) -- (0.8188,0.261) -- (0.8535,0.213) -- (0.887,0.166) -- (0.9195,0.1196) -- (0.944,0.083) -- (0.9698,0.045) -- (1,0.0) -- (1,-1) -- (-0.00,-0.99999) -- (-0.0634,-0.9576) -- (-0.11419,-0.923) -- (-0.15327,-0.896) -- (-0.2075,-0.858) -- (-0.26476,-0.816) -- (-0.307,-0.783) -- (-0.3527,-0.747) -- (-0.3945,-0.713) -- (-0.441,-0.673) -- (-0.487,-0.63) -- (-0.54,-0.58) -- (-0.587,-0.534) -- (-0.616,-0.503) -- (-0.653,-0.463) -- (-0.687,-0.424) -- (-0.715,-0.391) -- (-0.7463,-0.354) -- (-0.778,-0.314) -- (-0.814,-0.267) -- (-0.845,-0.225) -- (-0.879,-0.178) -- (-0.909,-0.135) -- (-0.937,-0.094) -- (-1,0) -- cycle;
\draw[line width=1.1pt,color=red] (-1,-1)--(-1,1)--(1,1)--(1,-1)--(-1,-1);
\draw[line width=1pt,color=magenta] (0,1)--(1,0);
\draw[line width=1pt,color=magenta] (0,-1)--(-1,0);
\draw[scale=1] (1.05,0) node[below] {\footnotesize $1 $};
\draw[scale=1] (0.05,1.1) node[left] {\footnotesize $1 $};
%\draw[scale=1] (1,0.0) node[below] {\footnotesize $1 $};
\draw[scale=1] (-1.1,0) node[below] {\footnotesize $-1 $};
\draw[scale=1] (0.05,-1.05) node[left] {\footnotesize $-1 $};
\draw[scale=1] (0.0,-1.4) node[below,rotate=0] {\footnotesize (ii) Real eigenvalues};
\end{tikzpicture}
\end{center}
\end{minipage}%
\caption{This figure illustrates the eigenvalues of a $2\times 2$ matrix as they appear within the unit disk $\mathbb{D}$, and as captured by our expansion technique. The colored regions in Part (i) of this diagram illustrate how non-real eigenvalues are accounted for in our expansion method. The inner gray-colored region belongs to $\|V_0\|_1<1$. The cyan-shaded region corresponds to $\|V_1\|_1<1$. In contrast, the area within the green curves reflects the non-real eigenvalues captured by $\|V_2\|_1<1.$    Part (ii) of the figure is similar but represents the case of real eigenvalues.}\label{Fig-Eigenvalues}
\end{figure}

To have the zeros of a polynomial $p$ inside $\mathbb{D},$ it is necessary to have $|p(0)|<1,$ $p(1)>0$ and $(-1)^np(-1)>0.$ It was clarified in \cite{Al-Ca-Ka2025} that $\|p\|_{\ell_1}<2$ implicitly achieves the three necessary conditions. Here, we can say more. Even if $\|p\|_{\ell_1}\geq 2,$ but $\|p_m\|_{\ell_1}<2$ for some $m,$ then the three necessary conditions are also satisfied as a consequence of Theorem \ref{Th-MainTheorem}. Having clarified this, it is advisable to check the three necessary conditions from the get-go, since it could be labor-saving when some roots are out of $\mathbb{D}.$

%%%%%%%%%%%%%%%%%%%%%%%%%%%%%%%%%%%%%%%%%%%%%%%%%%%%%%%%%%%%%
\subsection{The equilibria of the expanded systems}
%%%%%%%%%%%%%%%%%%%%%%%%%%%%%%%%%%%%%%%%%%%%%%%%%%%%%%%%%%%%%%

It is obvious that $\bar x=0$ is a fixed point of the linear map $F_m$ for all $m$; however, if $F_m$ has a fixed point, what does it mean for $F_{0}$? Before we address this question, observe that we obtain infinitely many fixed points of $F_0$ when
$$F_0(1,\ldots,1)=\sum_{j=0}^{k-1}a_j=1.$$
Also, $F_0(1,\ldots,1) $ is bounded by $-\|V_0\|_1$ and $ \|V_0\|_1.$
Similarly, a fixed point of $F_m$ means either $\bar x=0,$ or infinitely many fixed points when $F_m(1,\ldots,1)=1.$ This condition is a boundary component of the sufficient condition that we identified during our discussion of stability.
\begin{theorem}\label{Th-FixedPointsLinear}
Consider $F_0$ and $F_m$ as given in Eqs. \eqref{Eq-Linear1} and \eqref{Eq-Linear3}. Let $\alpha_k:=\sum_{j=0}^{k-1}a_j$ and $\beta_m=\sum_{j=0}^mb_{j,0}.$ Each of the following holds true:
\begin{description}
\item{(i)} If $\alpha_k=1,$ then for all $\bar x\in\mathbb{R},$ $\bar{x}$ is a fixed point of $F_m$ for all $m\in\mathbb{N}.$
\item{(ii)}  If $\alpha_k\neq 1,$ but $\beta_m=-1$ for some $m,$  then for all $\bar x\in\mathbb{R},$ $\bar{x}$ is a fixed point of $F_{m+1}.$
\end{description}
\end{theorem}
\begin{proof}
 From Eq. \eqref{Eq-Vm}, we have $V_{m+1}=J_0^tV_m,$ and consequently,
 \begin{align*}
 \sum_{j=0}^{k-1}b_{m+1,j}=&\left(\sum_{j=0}^{k-1}a_{j}-1\right)b_{m,0}+\sum_{j=0}^{k-1}b_{m,j}\\
 F_{m+1}(1,1,\ldots,1)=&(\alpha_k-1)b_{m,0}+F_m(1,1,\ldots,1).
 \end{align*}
 By induction on $m$, we obtain
$$F_{m+1}(1,1,\ldots,1)=(\alpha_k-1)(1+\beta_m)+1. $$
Therefore, if $\alpha_k=1,$ we obtain Part (i), and if $\beta_m=-1$ for some $m\in \mathbb{N},$ then we obtain Part (ii).
\end{proof}
It is worth mentioning that, unlike in the case of $F_0,$ if $F_m,\; m\geq 1$ has infinitely many fixed points, it does not mean that $F_{m+1}$ has infinitely many fixed points. We illustrate this issue in the following example.
\begin{example}
Consider $k=2$ in Eq. \eqref{Eq-Linear1}, i.e,
$$F_0(x_n,x_{n-1},x_{n-2})=a_0x_n+a_1x_{n-1}+a_2x_{n-2},$$
and let $a_1:=-(a_0^2+a_0+1).$ Based on the notations of Theorem \ref{Th-FixedPointsLinear}, we have
$$\beta_0=a_0,\;\beta_1=-1,\;\beta_2=a_2-(a_0+1)(a_0^2 +a_0+1)\quad \text{and}\quad \alpha_2=a_2-(a_0^2+1).$$
So, if $a_2\neq a_0^2 +2,$ then $\alpha_2\neq 1$ and $F_0$ has only the trivial fixed point. On the other hand, since $\beta_1=-1,$ $F_2$ has infinitely many fixed points, while $F_3$ has only the trivial fixed point as long as $a_2\neq (a_0+1)(a_0^2 +a_0+1)-1.$
\end{example}

%%%%%%%%%%%%%%%%%%%%%%%%%%%%%%%%%%%%%%%%%%%%%%%%%%
\section{Expansion and stability in nonlinear systems}\label{Sec-Nonlinear}
%%%%%%%%%%%%%%%%%%%%%%%%%%%%%%%%%%%%%%%%%%%%%%%%%%%%%
Expansion in the linear case is found to be fruitful in understanding the stability of the origin. This section focuses on Eq. \eqref{Eq-TheGeneralExpansion} when $F_0$ is nonlinear, and we aim to establish local and global stability results. We assume $F_0$ has a fixed point $\bar x,$ i.e., $F_0(x,\ldots,x)=x$ has a solution $x=\bar x.$ Also, we stress that $F_0\in \mathcal{C}^1(\mathbb{R}^k)$. For our writing convenience, we define $X_n=(x_n,\ldots,x_{n-k+1})$ and $Y_n=(x_n,\ldots,x_{n-k+2}).$ The one step ($m=1$) expansion becomes
\begin{equation}\label{Eq-NonLinear2}
x_{n+1}=F_1(X_{n-1})=F_0(F_0(X_{n-1}),Y_{n-1}).
\end{equation}
By repeating the process further (say $m$-times), the expanded system becomes
\begin{equation}\label{Eq-NonLinearm}
x_{n+1}=F_m(X_{n-m})=F_{m-1}(F_0(X_{n-m}),Y_{n-m}).
\end{equation}

Since a fixed point of $F_0$ stays a fixed point of $F_m$ for all $m,$ we investigate its local stability under $F_m$. At the equilibrium $\bar x,$ define $a_{i-1}=\frac{\partial F_0}{\partial x_i}$ for all $i=1,\ldots,k,$ and let $J_0$ denote the Jacobian matrix of the system. Based on this notation, $J_0$ becomes the same as $J_0$ in Section \ref{Sec-Linear}. Furthermore, implicit differentiation at an equilibrium $\bar x$ shows that the Jacobian matrix of $F_m$ is, in fact, the matrix $J_m$ obtained for the linear system in Eqs. \eqref{Eq-Linear3} and \eqref{Eq-The-b-System}. This fact and Theorem \ref{Th-MainTheorem} lead to the following local stability result:

\begin{theorem}\label{Co-Linearization}
Consider the expanded system in Eq. \eqref{Eq-NonLinearm}. The equilibrium solution $\bar x$ is locally asymptotically stable for $F_0$ if and only if $\|\nabla F_m\|_1$ at $\bar x_0$ is less than one for some  $m\in \mathbb{N}.$
\end{theorem}
\begin{proof}
Local stability of $F_0$ is obtained through linearization, and consequently, Theorem \ref{Th-MainTheorem} is applicable here.  Let $J_0$  be the Jacobian matrix of $F_0$ at the equilibrium $\bar x$ of $F_0,$ and $a_{i-1}=\frac{\partial F_0}{\partial x_i}$ for all $i=1,\ldots,k.$ The zero equilibrium of Eq. \eqref{Eq-Linear1} is locally stable if and only if there exists $m\in \mathbb{N}$ such that $\|(J_0^t)^mV_0\|_1<1,$ where $V_0$ is the column vector $[a_0,a_1,\ldots,a_{k-1}]^t.$ Now, the proof completes by observing that $\|(J_0^t)^mV_0\|_1=\|\nabla F_m\|_1.$ 
\end{proof}
Next, we discuss the expansion strategy's positive impact on global stability. We integrate the expansion strategy with the embedding technique, which necessitates a monotonicity assumption on $F_0$ or one of its expansions, $F_m$. Then, we depend on the recent results obtained by the authors in \cite{Al-Ca-Ka2024B}.  We consider a partial order on $\mathbb{R}^k$ that depends on the nature of the monotonicity in $F_0.$ $(x_1,\ldots,x_k)\leq_\tau (y_1,\ldots,y_k)$ means $x_i\leq y_i$ ($y_i<x_i$) whenever $F_0$ is increasing (decreasing) in its $i^{th}$-argument. Here, increasing is used to mean non-decreasing. Similarly, decreasing is used to mean non-increasing. $F_0$ becomes increasing under this partial order in the sense that $X\leq_\tau Y$ implies $F_0(X)\leq F_0(Y).$ Define the point $P_\tau=(u_1,\ldots,u_k)$ so that $u_i=x$ if $F_0$ is increasing in its $i^{th}$ argument, and $u_i=y$ otherwise. Also, consider $P_\tau^t$ to be the ``dual" of $P_\tau,$ where $x$ and $y$ are switched ($x\leftrightarrow y$). When $x\leq y,$ the points $P_\tau$ and $P_\tau^t$ represent the minimum and maximum (respectively) of a boxed region in $ k$-dimensional space equipped with the $\tau$-partial order.  Observe that  $(x,y)=(F_0(P_\tau), F_0(P^t_\tau))$ has at least one solution given by the fixed point  $x=y=\bar x.$ However, solutions in which $x\neq y$ are dubbed in \cite{Al-Ca-Ka2024B} as pseudo-fixed points of $F_0$. The same concepts generalize to tackle $F_m.$ The only change that we need to accommodate is the dimension. We cite the following result, then generalize it to $F_m.$
\begin{theorem}\cite{Al-Ca-Ka2024B} \label{Th-GlobalStability}
Let $F_0$ be monotonic in each of its arguments. Suppose $F_0$ has a unique fixed point and no pseudo-fixed points. If for each initial condition $X_0$ in the domain of $F_0,$ there exists a point $P_\tau$ such that 
\begin{equation}\label{theoremi}
x\leq y,\;\; x<F_0(P_\tau),\;\; F_0(P^t_\tau)<y \quad \text{and}\quad P_\tau\leq_\tau X_0 \leq P_\tau^t,
\end{equation}
then the fixed point of $F_0$ is a global attractor. 
\end{theorem}
Note that an expansion $F_m$ can change monotonicity when compared with $F_0$, and this change may lead to some advantages in proving global attractivity. We must clarify some technicalities that make the proof of Theorem \ref{Th-GlobalStability} valid for the generalized version we give in Theorem \ref{Th-GlobalStability-Main2}, which will fit our expansion strategy.  Assume that the delay in $F_0$ is $k$ and increases by one at each stage of the expansion process. This means that for each $m$, the delay in the $F_m$ system is $k+m.$ Therefore, the initial conditions must take the form $Y_0=(x_0,\ldots,x_{-k-m+1}).$ In this case, we can consider an auxiliary map $F_m^*$ that has $k+m$ arguments but it is constant in $m$ of them, i.e.,
$$F^*_m(Y_0)=F^*_m(x_0,\ldots,x_{-k-m+1})=F_m(x_{-m},\ldots,x_{-m-k+1}).$$
Therefore, $F_m^*:\; V^{k+m}\to V,$ where $V$ is  $\mathbb{R}$ or $\mathbb{R}_+.$ We embed the dynamical system obtained by $F^*_m$ (based on $F_m$) into the higher-dimensional dynamical system as defined in \cite{Al-Ca-Ka2024B}. Since $F_m^*$ is constant in its first $m$ arguments, we consider the monotonicity in the first $ m$ arguments to agree with the monotonicity of $F_0$ in its first argument. In this case, we have a $\tau$-partial order on $V^{k+m}.$ This enables us 
to give a generalized version of Theorem \ref{Th-GlobalStability}. We refer the reader to \cite{Al-Ca-Ka2024B} for more details about the machinery that makes our generalization possible.
\begin{comment}
increasing in those arguments, enabling us to define the $\tau$-partial order. Let $X=(x_1,x_2,\ldots, x_{k+m}), U=(u_1,u_2,\ldots, u_{k+m}).$ Now, define 
\begin{equation}\label{Eq-G}
    G:\; \mathbb{V}^{k+m}\times   \mathbb{V}^{k+m}\to \mathbb{V}^{k+m}\times   \mathbb{V}^{k+m}
\end{equation} as follows: If $F_m$ is increasing in its first argument, then 
$$G (X,U) =
(F_m(X), y_2,\ldots, y_{m+k}, F_m(U), z_2,\ldots z_{m+k}),$$
where for $1\leq i\leq m+k-1,$ we set $y_{i+1}=x_i$ and $z_{i+1}=u_i$ whenever $F_m$ has the same mononotonicy in its $(i,i+1)$-components. Otherwise, set 
$y_{i+1}=u_i$ and $z_{i+1}=x_i$. If $F_m$ is decreasing in its first argument, then define
$$G(X,U) =
(F_m(U), y_2,\ldots, y_{m+k}, F_m(X), z_2,\ldots z_{m+k}),$$
where the entries $2$ to $m+k$ are chosen as before. The map $G$ is monotonic with respect to the $\tau\times \tau$ partial order. Boundedness of orbits leads to convergence in the orbits of $G,$ which will squeeze the orbits of $F_m$ to converge. We refer the reader to \cite{Al-Ca-Ka2024B} for more details about this machinery that makes the following result possible. Also, we provide a glimpse of the proof since the details follow along the same lines as in \cite{Al-Ca-Ka2024B}.   
\end{comment}
\begin{theorem}\label{Th-GlobalStability-Main2}
Consider the map $F_0$ in Eq. \eqref{Eq-NonLinearm} and its expansions $F_m.$ Suppose that $F_0$ has a unique fixed point and for some $m\in \mathbb{N},$ $F_m$ is monotonic in each argument. If the conditions of Theorem \ref{Th-GlobalStability} are satisfied on $F_m,$ then $F_0$ has a global attractor. 
\end{theorem}
\begin{proof}
For some $m\in \mathbb{N},$ assume that $F_m$ is monotonic in each argument. Since the conditions of Theorem  \ref{Th-GlobalStability} are satisfied for $F_m,$ the proof follows along the proof of Theorem  \ref{Th-GlobalStability} to obtain bounded monotonic sequences that squeeze the orbits of $F_m$, and force convergence to the unique fixed point. Hence, $F_m$ has a globally attracting fixed point. Because every orbit of $F_0$ can be an orbit of $F_m$, the fixed point of $F_0$ is globally attracting. 
\end{proof}
In conclusion, we stress that the monotonicity character of $F_m$ may differ from that of $F_0,$ and this is the driving force behind Theorem \ref{Th-GlobalStability-Main2}. The key advantage is when the range of parameters that give global stability in $F_m$ exceeds that of $F_0.$ This will be further illustrated in the application section. 
%%%%%%%%%%%%%%%%%%%%%%%%%%%%%%%%%%%%%%%%%%%%%
\section{Applications}
%%%%%%%%%%%%%%%%%%%%%%%%%%%%%%%%%%%%%%%%%%%%%%%%%%

This section gives examples and models illustrating the advantages of our expansion strategy. Since local stability leads to investigating the spectrum of a Jacobian matrix, we begin with a simple algebraic example that illustrates the computational advantages of our approach.
\begin{example}\label{Ex-1}
\textbf{(An algebraic example):}
Consider the matrix
$$A= \begin{bmatrix}[1.5]
    \frac{7}{5} & -1 & -1 & -1 & 1 \\
    \frac{ 43157}{2400}& 1& 1& 2& -1 \\
    c& -2& 0& 0& 1 \\
   -\frac{139397}{2592}& 1& 1& -2& 0 \\
    c-\frac{82564}{2025}& 0& 1& -1& 1 \\
 \end{bmatrix}.$$
We investigate values of $c$ that make the spectrum of $A$ within $\mathbb{D}.$ The characteristic polynomials is given by
$$p_0(x)=x^5 - \frac{7}{5}x^4 + \frac{3509}{3600}x^3 - \frac{2933}{7200}x^2 +\frac{3199}{32400}x + \frac{299959}{16200}-c.$$
Note that we can apply our technique without checking any conditions; however, before indulging in the technique, it can be labor-saving \cite{Al-Ca-Ka2025} to test the simple necessary conditions $|det(A)|=|p(0)|<1,$ $p(1)>0$ and $(-1)^5p(-1)>0.$ These are the necessary conditions in the Jury's algorithm. If these necessary conditions lead to a vacuous domain of $c$, then there is no need to proceed.  We have
$$|det(A)|=\left|\frac{299959}{16200}-c\right|=|p(0)|<1\quad \Leftrightarrow\quad c\in I_0:=\left(\frac{283759}{16200},\frac{316159}{16200}\right)\approx (17.52,19.52).$$
$p(1)>0$ if $c<\frac{45077}{2400} \approx 18.78.$ $-p(-1)>0$ if $c>14.64.$
Therefore,  any possible value of $c$ must be within the interval $\tilde{I}_0:= \left(\frac{283759}{16200},\frac{45077}{2400}\right)\approx (17.52,18.78).$
Next, we proceed with our technique.   We have
$$\|p_0\|_{\ell_1}=\frac{251477}{64800} + \left|c-\frac{299959}{16200}\right|>2,$$ and consequently, Rouche's theorem is useless here.  Based on our notation in Section \ref{Sec-Linear}, we compute $V_{m+1}=(J_0^t)^mV_0,$ where
$$V_0= \begin{bmatrix}[1.5]
           \frac{7}{5} \\
           -\frac{3509}{3600} \\
           \frac{2933}{7200} \\
           -\frac{3199}{32400} \\
           c-\frac{299959}{16200} \\
        \end{bmatrix}
       \quad\text{and}\quad J_0^t=
                                           \begin{bmatrix}[1.5]
                                              \frac{7}{5} & 1 & 0 & 0 & 0 \\
                                              -\frac{3509}{3600} & 0 & 1 & 0 & 0 \\
                                               \frac{2933}{7200} & 0 & 0 & 1 & 0 \\
                                              -\frac{3199}{32400} & 0 & 0 & 0 & 1 \\
                                              c-\frac{299959}{16200} & 0 & 0 & 0 & 0 \\
                                            \end{bmatrix}.
 $$
We find
$$V_1= \begin{bmatrix}[1.5]
           \frac{3547}{3600} \\
           -\frac{3829}{4000} \\
           \frac{152789}{324000} \\
           c-\frac{3021983}{162000} \\
           \frac{7}{5}c-\frac{2099713}{81000} \\
        \end{bmatrix},\quad
        V_2= \begin{bmatrix}[1.5]
           \frac{15197}{36000} \\
           -\frac{2111621}{4320000} \\
           c-\frac{157704643}{8640000} \\
           \frac{7}{5}c-\frac{3034933573}{116640000} \\
           \frac{3547}{3600}c-\frac{1063954573}{58320000} \\
        \end{bmatrix},\quad
        V_3= \begin{bmatrix}[1.5]
           \frac{17659}{172800} \\
           c-\frac{1209447959}{64800000} \\
          \frac{7}{5}c- \frac{60297516251}{2332800000} \\
           \frac{3547}{3600}c-\frac{21327706663}{1166400000} \\
           \frac{15197}{36000}c-\frac{4558476923}{583200000} \\
        \end{bmatrix} .$$
$V_1$ and $V_2$ do not help, but $\|V_3\|_1$ can. Indeed, we obtain
$$\|V_3\|_1<1 \quad \text{for}\quad c\in I_3:=\left(\frac{860489737}{46994400},\frac{45077}{2400}\right)\approx (18.31,18.78).$$ This result can be improved further by computing $V_4;$ however, we skip the details. As we illustrated in Fig. \ref{Fig-Eigenvalues}, the location of the eigenvalues within $\mathbb{D}$ affects the speed of capturing them in our technique.
\\

It is worth mentioning that it is possible to rescale the unit disk and investigate the zeros inside a disk of radius $R$. In this case, you must handle the monic polynomial $R^{-5}p_0(Rx).$ For instance, from $\|V_1\|_1<1,$ it is easy to find that 
$$17.860\approx \frac{33561737}{1879200}<c<\frac{36297467}{1879200}\approx 19.315$$
will make all zeros of $p_0$ within a disk of radius $R=\frac{3}{2}.$ 
Similarly, for $\|V_3\|_1<1,$  $c\in [17.570,19.540]$ implies all zeros are located within a disk of radius $R=1.25.$ Finally, we note that running Jury's algorithm on this example results in complex expressions. The interested reader can compute and compare the complexity between the two approaches. 
\end{example}

\subsection{Ricker type models}
We consider a population model with delay $k=2$ and constant stocking in the form
\begin{equation}\label{Eq-Ricker1}
x_{n+1}=F_0(x_n,x_{n-1},x_{n-2})=x_nf(x_{n-2})+h,\quad h>0,
\end{equation}
where $h>0$ is a stocking or perturbation constant. The density dependent function $f$ is assumed to be decreasing, $f:\; [0,\infty)\to (0,\infty)$ and $\lim_{t\to \infty}tf(t)=0$. A typical choice of this map is $f(t)=e^{b-t},$ which represents the Ricker model \cite{Al-Ka2024,Al-Ca-Ka2024B}. Therefore, we consider $f(0)=b>0$ throughout our discussion on this model.  At this stage, we keep $f$ in general, but we assume it to be differentiable.
The assumptions on $f$ in System \eqref{Eq-Ricker1} guarantee the existence of a unique positive equilibrium that we denote by $\bar x.$ The existence follows from the fact that $h>0$ and $\lim tf(t)=0.$ On the other hand, uniqueness can be verified as follows: Let 
$$\bar x=\bar xf(\bar x)-h\quad \text{and}\quad \bar y=\bar y f(\bar y)+h,\;\;\bar x\neq \bar y.$$
Subtract to obtain 
\begin{align*}
1=&\bar y\frac{f(\bar y)-f(\bar x)}{\bar y-\bar x}+f(\bar x)\\
=&\bar y f^\prime(\xi)+f(\bar x)\\
=&\bar y f^\prime(\xi)+1-\frac{h}{\bar x}<1,
\end{align*}
which is not possible. Observe that $\bar x>h$ and $f(\bar x)=1-\frac{h}{\bar x}<1.$
\\

Now, we focus on local stability. Computing the eigenvalues of the Jacobian matrix is a formidable task. So, we do it through our established expansion technique. The function $F_0$ depends on three variables, say $(x,y,z).$ At the equilibrium $\bar x$, we have
\begin{align*}
a_0=&\frac{\partial F_0}{\partial x}=f(\bar x)=1-\frac{h}{\bar x}\\
a_1=&\frac{\partial F_0}{\partial y}=0\\
a_2=&\frac{\partial F_0}{\partial z}=\bar xf^\prime(\bar x).
\end{align*}
Therefore, we have $V_0=(a_0,0,a_2),$ where $0<a_0<1$ and $a_2<0.$ We obtain our local stability conditions from $\|V_m\|_1<1,$ where
$$V_{m+1}=(J_0^t)^mV_0\quad \text{and}\quad
J_0^t=\begin{bmatrix}[1.5]
           1-\frac{h}{\bar x}&1&0 \\
           0&0&1 \\
           \bar xf^\prime(\bar x)&0&0 \\
           \end{bmatrix}.$$
Since $\|V_0\|_1=\left(1-\frac{h}{\bar x}\right)+\bar x |f^\prime(\bar x)|.$ This gives us
\begin{equation}\label{In-V0}
\|V_0\|_1<1 \quad \Leftrightarrow\quad a_0+|a_2|<1 \quad \Leftrightarrow\quad|f^\prime(\bar x)|<\frac{h}{\bar x^2}.
\end{equation}
 Next, we have
$$\|V_1\|_1<1 \quad \Leftrightarrow\quad (a_0+1)(a_0-1+|a_2|)<0,$$
which does not improve the inequality obtained from $\|V_0\|_1<1$ since $a_0+1>0.$ So, we proceed to $V_2$ and $V_3.$ We have
\begin{equation}\label{In-V2}
\|V_2\|_1<1 \quad \Leftrightarrow\quad |a_0^3+a_2|+a_0(1+a_0)|a_2|<1
\end{equation}
and
\begin{equation}\label{In-V3}
\|V_3\|_1<1 \quad \Leftrightarrow\quad a_0|a_0^3+2a_2|+|a_2(a_0^3+a_0^2+a_2)|<1.
\end{equation}
The conditions obtained in inequalities \eqref{In-V0}, \eqref{In-V2} and \eqref{In-V3} are all sufficient conditions for the local stability of $\bar x.$ However, determining the necessary conditions is doable because the characteristic polynomial of the matrix $J_0$ is of degree three. Indeed, we obtain
$$p(t)=t^3-a_0t^2-a_2.$$
Since $a_0>0$ and $a_2<0,$  there is one negative root. To have the negative root larger than -1, we need $p(-1)>-1.$ The other two roots are complex or real and symmetric around $\frac{2}{3}a_0.$ Therefore, we must have $p(1)>0.$ Put both conditions together to obtain
$$-1<a_0+a_2<1.$$
The remaining case is when two roots are non-real, and in this case, isolate the real and imaginary parts of $p(x+iy)$ to obtain
$$(a_0 - 3x)y^2 =a_2+a_0x^2- x^3\quad and \quad y^2=3x^2-2a_0x.$$
Algebraic manipulations between these two equations and the inequality $x^2+y^2<1$ give us a new condition on the parameters, which is given by
$$ a_2(a_2-a_0)<1.$$
Hence, the eigenvalues are in $\mathbb{D}$ if
\begin{equation}\label{In-LocalStability-Ricker}
|a_0+a_2|<1\quad \text{and}\quad a_2(a_2-a_0)<1,\quad \text{where}\quad 0<a_0<1,\; a_2<0.
\end{equation}
It is possible to consider a larger delay in Eq. \eqref{Eq-Ricker1}. Our algorithm works similarly in this case, but computing the eigenvalues of $J_0$ becomes challenging. However, the main objective here is to illustrate the implementation of our developed theory. Part (i) of Fig. \ref{Fig-LocalStabilityRicker} illustrates the stability regions in the $(-a_2,a_0)$-plane based on the obtained conditions in \eqref{In-V0} to \eqref{In-V3}. For the specific choice $f(t)=e^{b-t}$ (i.e., the Ricker model), we can translate the stability region into the $(h,b)$-plane as given in Part (ii) of the same figure. In this case, we have $b=\bar x+\ln\left(1-\frac{h}{\bar x}\right),$ where $\bar x$ and $h$ can be obtained from the constraints on $a_0$ and $a_2.$ Indeed, $\bar x=\frac{-a_2}{a_0}$ and $h=\frac{a_2(a_0-1)}{a_0}.$ $a_0=1$ gives $h=0$ in Part (ii) of the figure. The boundary condition $a_2(a_2-a_0)=1$ gives 
$$\left(\bar x-h\right)^2(1+\bar x)=\bar x, \quad\text{where}\quad (\bar x-h)=\bar xe^{b-\bar x}.$$
The curve is labeled as ``LC" in Part (ii) of the figure, which forms the upper boundary of the local stability region. Note that an explicit expression in $b$ and $h$ for the LC curve is not impressive, but it can be found based on the fact that $\bar x>h.$   On the other hand, the curve $a_0-a_2=1$ (obtained from $\|V_0\|_1<1$) translates into 
$$\bar x= \frac{1}{2}(h+h_0),\quad \text{where}\quad  h_0:=\sqrt{h^2+4h}\quad \text{and}\quad (\bar x-h)=\bar xe^{b-\bar x} .$$
 This is the curve $b_\infty$ found in \cite{Al-Ca-Ka2024B}, which gives us the boundary of the global stability region in the $(h,b)$-plane as found in \cite{Al-Ca-Ka2024B}, i.e.,
$$(h,b):\; b<b_\infty=\ln\left(1+\frac{1}{2}(h-h_0)\right)+\frac{1}{2}(h+h_0).$$
Observe that $b$ is increasing in $\bar x,$ and consequently, the stability region extends when conditions \eqref{In-V2} or \eqref{In-V3} give a larger value of $\bar x.$ Explicit expressions become challenging to obtain, but Fig. \ref{Fig-LocalStabilityRicker} clarifies the relationship between the regions in the $(-a_2,a_0)$-plane, and the $(h,b)$-plane when $f(t)=\exp(b-t)$.
%%%%%%%%%%%%%%%%%%%%%%%%%%%%%%%%%%%%%%%%%%%%%%%%%%
\definecolor{MyMaroon}{rgb}{128,0,0}
\definecolor{MyOrange}{rgb}{255,87,51}
\definecolor{ffvvqq}{rgb}{1,0.3333333333333333,0}
\definecolor{qqqqff}{rgb}{0.,0.,1.}
\definecolor{ffqqqq}{rgb}{1.,0.,0.}
\definecolor{cqcqcq}{rgb}{0.7529,0.7529,0.7529}
\begin{figure}[htbp]
\centering
\begin{minipage}[t]{0.5\textwidth}
%\raggedright
\begin{center}
\begin{tikzpicture}[line cap=round,line join=round,>=triangle 45,x=1.0cm,y=1.0cm,scale=3.0]
\draw[line width=1.0pt,color=red,dashed] (0,1)--(0.5,1.5);
\draw[line width=1.2pt,color=red] (0,1)--(0.62,1.0);
\fill[line width=1pt,color=ffqqqq,fill=olive!30,fill opacity=1.0] (0.0025587402831601085,1.0000016325977032) -- (0.5858059087831621,0.9994632793509961) -- (0.5899528275216583,0.9576440714798042) -- (0.6048646372937606,0.9027431443888763) -- (0.6221692151330205,0.8671207167073338) -- (0.6038873789867656,0.7986343361739312) -- (0.5862905399306916,0.7302053566354707) -- (0.5808046107277185,0.6725922033753593) -- (0.5857790740502431,0.5978507832930657) -- (0.5945854723046727,0.5547682521976696) -- (0.6072026053412964,0.5126969719562907) -- (0.6233678468219132,0.47125862342625363) -- (0.6521355485052627,0.41288157479683274) -- (0.6823208352137107,0.3624030219913857) -- (0.71971960558292,0.30826198657688897) -- (0.7601012349273998,0.2561817327828609) -- (0.8044024739423712,0.20390316506255196) -- (0.8460733301029869,0.1578084065321847) -- (0.8792634085945888,0.12256402632525315) -- (0.9116980850960654,0.08900420219866972) -- (0.9574747814695552,0.04260247017742287) -- (0.991875285143704,-0.008124716853545504) -- (0,0) -- cycle;
\fill[line width=1pt,color=ffqqqq,fill=cyan!50,fill opacity=1.0] (0.0028434137191456757,0.9990539913582257) -- (0.04566951498734224,0.9852309515094558) -- (0.0975348613810666,0.9695117414553143) -- (0.15043629756675003,0.9545536922557958) -- (0.20419105372645008,0.9403840260793799) -- (0.25562924058848535,0.9277231151613342) -- (0.2998616357551912,0.9174869476352272) -- (0.33977239966190254,0.9087329017071144) -- (0.3974431702001321,0.8968342755815959) -- (0.48172147904841917,0.8808998523288495) -- (0.5555666144211551,0.8682077604379559) -- (0.6156197882502316,0.8586676291803916) -- (0.6193612594748579,0.7975149845398417) -- (0.6138078365340238,0.7296455264273286) -- (0.6133178964279453,0.6900350503823172) -- (0.6157859086552143,0.6401668835020284) -- (0.6247523714310479,0.5737684389199891) -- (0.6396352205169472,0.5108467457581254) -- (0.6642625901714185,0.44006682921527973) -- (0.6962507124165729,0.37184098463974397) -- (0.740318204678145,0.29776763331317324) -- (0.78403715125681,0.23641194025317921) -- (0.8077649727394165,0.2062950108502536) -- (0.8459403066862821,0.1611047063965118) -- (0.898700718517722,0.10327706513952394) -- (0.9361520143665008,0.06434668340213284) -- (0.9675252691736914,0.03254140020672208) -- (0.9965763455262295,-0.003423653412515565) -- (0.8,0) -- (0.6,0) -- (0.4,0) -- (0.2,0) -- (0,0) -- cycle;
\draw[line width=1pt,color=magenta] (0.0028,0.999) -- (0.046,0.985) -- (0.0975,0.9695) -- (0.1504,0.95455) -- (0.204,0.9404) -- (0.2556,0.9277) -- (0.29986,0.9175) -- (0.33977,0.9087) -- (0.39744,0.8968) -- (0.4817,0.881) -- (0.556,0.868) -- (0.6156,0.858667) -- (0.61936,0.7975) -- (0.6138,0.7296) -- (0.6133,0.69) -- (0.6158,0.6402) -- (0.62475,0.574) -- (0.6396,0.511) -- (0.664,0.44) -- (0.696,0.37184) -- (0.7403,0.298) -- (0.784,0.2364) -- (0.80776,0.206) -- (0.846,0.1611) -- (0.8987,0.1033) -- (0.9362,0.064) -- (0.9675,0.033) -- (0.997,0);
\draw[line width=1.2pt,domain=0.5:0.62,smooth,variable=\x,color=red,dashed] plot ({\x},{(1/\x-\x)});
\draw[line width=1.2pt,domain=0.62:1.0,smooth,variable=\x,color=red] plot ({\x},{(1/\x-\x)});
\fill[fill=gray!20] (0,0)--(1,0)--(0,1)--(0,0);
\draw[line width=1.0pt,color=blue] (0,1)--(1,0);
%\fill[fill=olive!30] (0,0)--(1,0)--(0,-1)--(0,0);
\draw[-triangle 45, line width=1.0pt,scale=1] (0,0) -- (1.5,0) node[below] {$-a_2$};
\draw[line width=1.0pt,-triangle 45] (0,0) -- (-0.2,0);
\draw[-triangle 45, line width=1.0pt,scale=1] (0,0) -- (0,1.5) node[left] {$a_0$};
\draw[line width=1.0pt,-triangle 45] (0,0) -- (0.0,-0.2);
\draw[scale=1] (1,0) node[below] {\footnotesize $1 $};
\draw[scale=1] (0,1.0) node[left] {\footnotesize $1 $};
\draw[scale=1] (0.6,-0.3) node[below] {\footnotesize (i) The general case};
\end{tikzpicture}
\end{center}
\end{minipage}%
%%%%%%%%%%%%%%%%%%%%%%%%%%%%%%%%%%%%%%%%%%%%%%%%%%%%%%%%%%%%%%%%%%%%%%%%%%%%%%%%%%%%%%%%%%%%
\begin{minipage}[t]{0.5\textwidth}
%\raggedright
\begin{center}
\begin{tikzpicture}[line cap=round,line join=round,>=triangle 45,x=1.0cm,y=1.0cm,scale=1.0]
\draw[-triangle 45, line width=1.0pt,scale=1] (0,0) -- (0,4.5) node[left] {$b$};
\draw[line width=1.0pt,-triangle 45] (0,0) -- (0.0,-0.5);
\draw[line width=0.8pt,color=green,dashed,domain=0:4.0] plot(\x,{\x});
\draw[scale=1] (0.5,0.6) node[above,rotate=20] {\scriptsize $LC$};\draw[line width=1.0pt,color=red,smooth]
(0.0,0.6)--(0.05,0.613198)--(0.1,0.614348)--(0.15,0.619992)--(0.2,0.629162)--(0.25,0.641194)--(0.3,0.655605)--(0.35,0.672036)--(0.4, 0.690209)--(0.5,0.730946)--(0.6,0.77651)--(0.7,0.826004)--(0.8,0.878782)--(0.85,0.906245)--(0.95, 0.963071)--
(1.,0.992351)--(1.1,1.05247)--(1.2,1.11447)--(1.3,1.17815)--(1.4,1.24335)--(1.5,1.30993)--(1.6,1.37777)--(1.7,1.44676)--(1.75,1.48166)--(1.8,1.51682)--(1.9,1.58786)--(1.95,1.62373)--(2.,1.65982)--(2.1,1.73263)--(2.15,1.76933)--(2.2,1.80623)--
(2.3,1.88059)--(2.35,1.91803)--(2.4,1.95565)--(2.5,2.03137)--(2.6,2.10771)--(2.7,2.18465)--(2.8,2.26214)--(2.9,2.34017)--(3.,2.4187)--(3.1,2.49772)--(3.2,2.57719)--(3.3,2.6571)--(3.4,2.73743)--(3.5,2.81816)--(3.6,2.89927)--(3.7,2.98075)--(3.8,3.06259)--
(3.9,3.14476)--(3.95,3.18597)--(4.,3.22726)--(4.1,3.31008)--(4.15,3.3516)--(4.2,3.39319)--(4.3,3.4766)--(4.35,3.51841)--(4.4,3.56029)--(4.45,3.60224)--(4.5,3.64425)--(4.6,3.72848)--(4.7,3.81295)--(4.8,3.89768)--(4.9,3.98264)--(5.,4.06783);
\fill[fill=gray!20] (5.0,3.937)--(4.97,3.9)--
(4.939,3.876)--(4.9,3.845)--(4.869,3.81)--(4.83,3.78)--(4.79,3.75)--(4.76,3.72)--
(4.7288,3.693)--(4.69,3.66)--(4.6587,3.63)--(4.62,3.6)--(4.588679,3.57)--(4.5536358173076925,3.54)--
(4.52,3.5)--(4.4835,3.48)--(4.4485,3.45)--(4.41,3.42)--(4.378,3.39)--(4.34,3.36)--
(4.31,3.33)--(4.27,3.2997)--(4.24,3.3)--(4.2,3.239)--(4.168,3.2)--(4.1,3.2)--
(4.098,3.15)--(4.063,3.119)--(4.03,3.089)--(3.99,3.0596)--(3.95,3.0297)--(3.9,2.9998)--
(3.88,2.97)--(3.85,2.94)--(3.8177,2.91)--(3.78,2.88)--(3.747,2.85)--(3.7,2.82)--
(3.6775,2.79)--(3.64,2.76)--(3.607,2.73)--(3.57,2.7028)--(3.537,2.67)--(3.502,2.64)--
(3.467,2.6143)--(3.43,2.58)--(3.397,2.55)--(3.362,2.53)--(3.34,2.497)--(3.29,2.467)--
(3.257,2.438)--(3.22198,2.40888)--(3.1869,2.379)--(3.15,2.35)--(3.1168,2.32)--(3.0818,2.29)--
(3.04676,2.263)--(3.012,2.234)--(2.976,2.205)--(2.941,2.176)--(2.90,2.147)--(2.87,2.12)--
(2.84,2.0896)--(2.8,2.06)--(2.766,2.03)--(2.731,2.0034)--(2.696,1.97)--(2.66,1.946)--
(2.626,1.91755)--(2.59,1.889)--(2.556,1.86)--(2.52,1.83)--(2.486,1.8037)--(2.451,1.775)--(2.42,1.747)--(2.38,1.7188)--(2.35,1.69)--(2.31,1.66)--(2.275,1.63)--(2.24,1.606)--
(2.2057,1.578)--(2.17,1.55)--(2.1356,1.52)--(2.10,1.49466)--(2.0655,1.46688)--(2.03,1.439)--
(1.995,1.4115)--(1.96,1.3839)--(1.925,1.356)--(1.89,1.328889)--(1.855,1.30)--(1.82,1.274)--
(1.785,1.2468)--(1.75,1.2196)--(1.715,1.19)--(1.68,1.165)--(1.645,1.138)--(1.61,1.1114)--
(1.5749,1.0845)--(1.54,1.0577)--(1.5048,1.031)--(1.4698,1.004)--(1.434758,0.9778)--(1.3997,0.95)--
(1.36467,0.9249)--(1.3296,0.8986)--(1.29,0.87)--(1.259,0.846)--(1.224,0.82)--(1.189,0.794)--
(1.154,0.768)--(1.119,0.74)--(1.084,0.7169)--(1.049,0.69)--(1.014,0.666)--(0.979,0.64)--
(0.944,0.615)--(0.92,0.59)--(0.87,0.565)--(0.839,0.54)--(0.8,0.52)--(0.7689,0.49)--
(0.73,0.466)--(0.6988,0.44)--(0.66,0.42)--(0.63,0.39)--(0.594,0.37)--(0.558,0.346)--
(0.52,0.32)--(0.488,0.299)--(0.45,0.276)--(0.418,0.253)--(0.38,0.23)--(0.35,0.21)--
(0.31,0.185)--(0.278,0.163)--(0.243,0.141)--(0.208,0.1197)--(0.173,0.098)--(0.138,0.077)--
(0.103,0.057)--(0.068,0.0369)--(0.033,0.0)--(0,0)--(5.0,0)--(5.0,3.937);
\fill[fill=cyan!50] (5.0,3.937)--(4.97,3.9)--
(4.939,3.876)--(4.9,3.845)--(4.869,3.81)--(4.83,3.78)--(4.79,3.75)--(4.76,3.72)--
(4.7288,3.693)--(4.69,3.66)--(4.6587,3.63)--(4.62,3.6)--(4.588679,3.57)--(4.55,3.54)--
(4.52,3.5)--(4.4835,3.48)--(4.4485,3.45)--(4.41,3.42)--(4.378,3.39)--(4.34,3.36)--
(4.31,3.33)--(4.27,3.2997)--(4.24,3.3)--(4.2,3.239)--(4.168,3.2)--(4.1,3.2)--
(4.098,3.15)--(4.063,3.119)--(4.03,3.089)--(3.99,3.0596)--(3.95,3.0297)--(3.9,2.9998)--
(3.88,2.97)--(3.85,2.94)--(3.8177,2.91)--(3.78,2.88)--(3.747,2.85)--(3.7,2.82)--
(3.6775,2.79)--(3.64,2.76)--(3.607,2.73)--(3.57,2.7028)--(3.537,2.67)--(3.502,2.64)--
(3.467,2.6143)--(3.43,2.58)--(3.397,2.55)--(3.362,2.53)--(3.34,2.497)--(3.29,2.467)--
(3.257,2.438)--(3.22198,2.40888)--(3.1869,2.379)--(3.15,2.35)--(3.1168,2.32)--(3.0818,2.29)--
(3.04676,2.263)--(3.012,2.234)--(2.976,2.205)--(2.941,2.176)--(2.90,2.147)--(2.87,2.12)--
(2.84,2.0896)--(2.8,2.06)--(2.766,2.03)--(2.731,2.0034)--(2.696,1.97)--(2.66,1.946)--
(2.626,1.91755)--(2.59,1.889)--(2.556,1.86)--(2.52,1.83)--(2.486,1.8037)--(2.451,1.775)--(2.42,1.747)--(2.38,1.7188)--(2.35,1.69)--(2.31,1.66)--(2.275,1.63)--(2.24,1.606)--
(2.2057,1.578)--(2.17,1.55)--(2.1356,1.52)--(2.10,1.49466)--(2.0655,1.46688)--(2.03,1.439)--
(1.995,1.4115)--(1.96,1.3839)--(1.925,1.356)--(1.89,1.328889)--(1.855,1.30)--(1.82,1.274)--
(1.785,1.2468)--(1.75,1.2196)--(1.715,1.19)--(1.68,1.165)--(1.645,1.138)--(1.61,1.1114)--
(1.5749,1.0845)--(1.54,1.0577)--(1.5048,1.031)--(1.4698,1.004)--(1.434758,0.9778)--(1.3997,0.95)--
(1.36467,0.9249)--(1.3296,0.8986)--(1.29,0.87)--(1.259,0.846)--(1.224,0.82)--(1.189,0.794)--
(1.154,0.768)--(1.119,0.74)--(1.084,0.7169)--(1.049,0.69)--(1.014,0.666)--(0.979,0.64)--
(0.944,0.615)--(0.92,0.59)--(0.87,0.565)--(0.839,0.54)--(0.8,0.52)--(0.7689,0.49)--
(0.73,0.466)--(0.6988,0.44)--(0.66,0.42)--(0.63,0.39)--(0.594,0.37)--(0.558,0.346)--
(0.52,0.32)--(0.488,0.299)--(0.45,0.276)--(0.418,0.253)--(0.38,0.23)--(0.35,0.21)--
(0.31,0.185)--(0.278,0.163)--(0.243,0.141)--(0.208,0.1197)--(0.173,0.098)--(0.138,0.077)--
(0.103,0.057)--(0.068,0.0369)--(0.033,0.0)--(0,0)--(0.0002,0.0165)--(0.0009,0.0371)--(0.0025,0.0617)--(0.0054,0.0946)--(0.0094,0.1297)--(0.0141,0.1639)--
(0.0176,0.1865)--(0.0230,0.2184)--(0.0285,0.2489)--(0.0349,0.2823)--(0.0461,0.3361)--(0.0550,0.3764)--(0.0646,0.4178)--(0.0757,0.4637)--(0.0836,0.4958)--(0.0915,0.5269)--(0.1028,0.5701)--
(0.1434,0.5571)--(0.1698,0.5448)--(0.2115,0.5302)--(0.2605,0.5197)--(0.3069,0.5156)--
(0.3571,0.5165)--(0.4015,0.5212)--(0.4757,0.5361)--(0.5599,0.5612)--(0.6431,0.5927)--(0.7617,0.6462)--(0.9192,0.7288)--(1.0586,0.8101)--(1.2345,0.9206)--(1.3991,1.0300)--(1.6100,1.1769)--
(1.9375,1.4161)--(2.5211,1.8653)--(3.1076,2.3368)--(3.9623,3.0484)--(4.9383,3.8840)--(5.0,3.937);
\draw[line width=1pt,color=blue,smooth,samples=100,domain=0.02:5.0] plot(\x,{ln((\x*\x+4*\x)^0.5-\x)-ln(\x+(\x*\x+4*\x)^0.5)+\x/2+(\x*\x+4*\x)^0.5/2});
\draw[line width=1pt,color=magenta,smooth]
(0.0000,0.0000)--(0.0002,0.0165)--(0.0009,0.0371)--(0.0025,0.0617)--(0.0054,0.0946)--(0.0094,0.1297)--(0.0141,0.1639)--
(0.0176,0.1865)--(0.0230,0.2184)--(0.0285,0.2489)--(0.0349,0.2823)--(0.0461,0.3361)--(0.0550,0.3764)--(0.0646,0.4178)--(0.0757,0.4637)--(0.0836,0.4958)--(0.0915,0.5269)--(0.1028,0.5701)--(0.1434,0.5571)--(0.1698,0.5448)--(0.2115,0.5302)--(0.2605,0.5197)--(0.3069,0.5156)--
(0.3571,0.5165)--(0.4015,0.5212)--(0.4757,0.5361)--(0.5599,0.5612)--(0.6431,0.5927)--(0.7617,0.6462)--(0.9192,0.7288)--(1.0586,0.8101)--(1.2345,0.9206)--(1.3991,1.0300)--(1.6100,1.1769)--(1.9375,1.4161)--(2.5211,1.8653)--(3.1076,2.3368)--(3.9623,3.0484)--(4.9383,3.8840);
\draw[-triangle 45, line width=1.0pt,scale=1] (0,0) -- (5.3,0) node[below] {$h$};
\draw[line width=1.0pt,-triangle 45] (0,0) -- (-0.5,0);
\draw[scale=1] (4.5,3.5) node[below,rotate=39] {\footnotesize $b_\infty$};
\draw[scale=1] (0,1) node[left,rotate=0] {\scriptsize $1$};
\draw[scale=1] (3.8,3.7) node[above,rotate=45] {\footnotesize $b=h$};
\draw[scale=1] (1.5,3.0) node[above,rotate=30] {\footnotesize Instability};
\draw[scale=1] (2.5,-0.3) node[below] {\footnotesize (ii) The case $f(t)=e^{b-t}$};
\end{tikzpicture}
\end{center}
\end{minipage}%
\caption{Part (i) of this figure shows the stability regions in the $(-a_2,a_0)$-plane obtained from the inequalities in \eqref{In-V0} to \eqref{In-V3}. In particular, the grey shaded region belongs to $\|V_0\|_1<1,$ the cyan shaded region belongs to  $\|V_1\|_1<1,$ and the third shaded region belongs to  $\|V_2\|_1<1.$ Note that there is an overlap between the regions. The connected red curve represents the boundary of the local stability region obtained by the conditions in \eqref{In-LocalStability-Ricker}. Part (ii) of the figure is similar but reflects the regions in the $(h,b)$-plane for the particular case $f(t)=e^{b-t}$. The region that belongs to  $\|V_2\|_1<1$ in Part (ii) was left unshaded to avoid overcrowding.}\label{Fig-LocalStabilityRicker}
\end{figure}
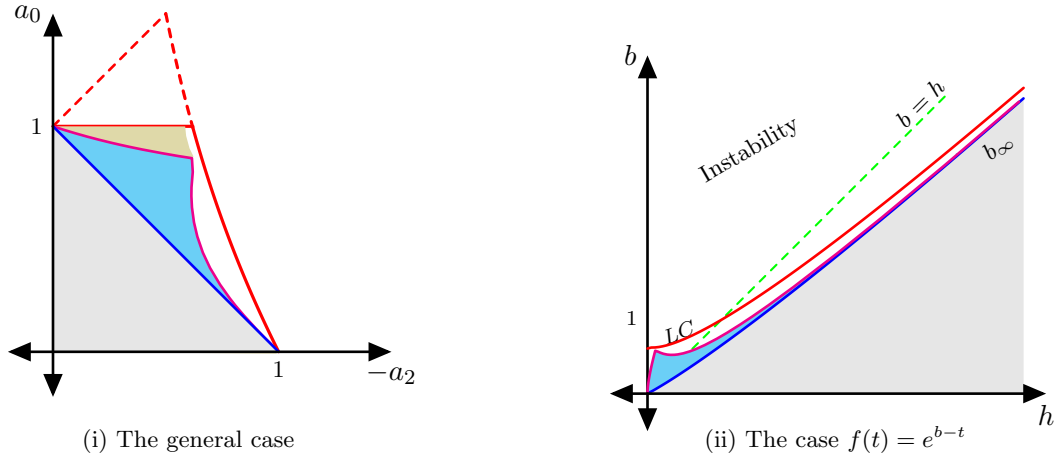
%%%%%%%%%%%%%%%%%%%%%%%%%%%%%%%%%%%%%%%%%%%%%%%%%%%%%%%%%%%%%%%%%%%%%%%%%%%%%%%%%%%%%%%%%%%%%%%%%%%%%%%

Next, we proceed to global stability based on our expansion strategy. When we use our process to increase the delay by one unit, i.e., $m=1$ in the above process, we obtain
\begin{equation}\label{Eq-Ricker2}
x_{n+1}=F_1(x_{n-1},x_{n-2},x_{n-3})= (x_{n-1}f(x_{n-3})+h)f(x_{n-2})+h.
\end{equation}
By Proposition \ref{Pr-Equilibrium}, $F_1$ has a unique fixed point, which is the fixed point of $F_0.$ The rule of thumb here is to expand till we reach a desired form of monotonicity.  Observe that $F_1$ is decreasing in $x_{n-2}$ and $x_{n-3}.$  So, we expand again to obtain
\begin{equation}\label{Eq-Ricker3}
\begin{split}
x_{n+1}=&F_2(x_{n-2},x_{n-3},x_{n-4})\\
=& x_{n-2}f(x_{n-2})f(x_{n-3})f(x_{n-4})+hf(x_{n-2})f(x_{n-3})+hf(x_{n-2})+h.
\end{split}
\end{equation}
At this point, the decreasing terms started to impact the increasing component. This can be observed by finding the derivative with respect to $x_{n-2}$. The negative terms in the derivative have the potential to neutralize the positive terms. In fact, since $x_n>h$ for all $n\geq 1,$ a mild condition like $h\geq 1$ ensures that $F_2$ is decreasing in all components. Repeating the expansion in the same way complicates the obtained expressions. So, we stop at $F_2,$ and we establish the following fact about $F_2.$  
\begin{proposition}\label{Pr-Ricker}
$F_2$ has a unique fixed point, which is the fixed point of $F_0.$ 
\end{proposition}
\begin{proof}
This result comes as a consequence of Proposition \ref{Pr-Equilibrium}. Observe that $F_0(x,y,z)=xf(z)+h$ and 
$$F_1(x,y,z)=xf(y)f(z)+hf(y)+h.$$
Both $F_0$ and $F_1$ are increasing in the first component. Thus, $F_0,F_1$ and $F_2$ have the same fixed point. 
\end{proof}
We need the following lemma to ensure the monotonicity of $F_2$ in all its arguments. 
\begin{lemma}\label{Lem-Ricker}
Consider $f$ as given in Eq. \eqref{Eq-Ricker1} and define $q(x)=x+\frac{f(x)}{f^\prime(x)}.$ If $q(x)>-\frac{h}{b}\left(1+\frac{1}{b}\right)$ for all $x,$ then $F_2$ is decreasing in $x_{n-2}.$
\end{lemma}
\begin{proof}
From \eqref{Eq-Ricker3}, we have
$$F_2(x,y,z)=xf(x)f(y)f(z)+hf(x)f(y)+hf(x)+h,$$
and we obtain 
$$\frac{\partial F_2}{\partial x}=f^\prime(x)\left[q(x)f(y)f(z)+h(1+f(y))\right].$$
Since $f^\prime(x)<0,$ we focus on the sign of the bracketed expression. It is positive if $q(x)\geq 0.$ So, consider $q(x)<0.$ In this case, used the fact that $f(z)\leq b$ to obtain 
$$q(x)f(y)f(z)+h(1+f(y))\geq f(y)\left(bq(x)+h)\right)+h.$$
If $q(x)\geq -\frac{h}{b}$, then the right-hand side of the inequality is positive. However, if $q(x)< -\frac{h}{b},$ then we again use the fact that $f(y)\leq b$ to obtain 
$$f(y)\left(bq(x)+h)\right)+h>b^2q(x)+h(b+1),$$
and the condition of $q$ makes the right-hand side of the inequality positive. Hence, the proof is complete.
\end{proof}

\begin{remark}
When $f(t)=\exp(b-t),$ $h>\frac{b^2}{b^2+b+1}$ is sufficient to make $F_2$ decrease in all its arguments. This follows from the fact that $q(x)=\bar x-1$ in Lemma \ref{Lem-Ricker} and the fact that $\bar x>h.$
\end{remark}
 Now, we are in a position to provide the following global stability result:
\begin{theorem}\label{Th-RickerGlobalStability}
Consider $F_0$ as given in Eq. \eqref{Eq-Ricker1} and $F_2$ as given in Eq. \eqref{Eq-Ricker3}. Let $g(x)=F_2(x,x,x)$ and $q(x)$ as defined in Lemma \ref{Lem-Ricker}. If $g$ has no $2$-cycle and $q(x)>-\frac{h}{b}\left(1+\frac{1}{b}\right),$  then the equilibrium $\bar x$ of Eq. \eqref{Eq-Ricker1}  is globally attracting.
\end{theorem}
\begin{proof}
By Proposition \ref{Pr-Ricker}, $F_2$ and $F_0$ share the unique fixed point $\bar x.$ Showing that $\bar x$ is globally attracting for $F_2$ implies it is globally attracting for $F_0.$ Next, we want to apply Theorem \ref{Th-GlobalStability-Main2} on $F_2.$ $F_2$ is constant in $x_n$ and $x_{n-1};$ however, we consider $F_2$ as non-increasing in both arguments. Also, $F_2$ is decreasing in all other arguments by Lemma \ref{Lem-Ricker}. Because the one-dimensional map $g$ has no $2$-cycle, the system $(y,x)=(F_2(x,x,x),F_2(y,y,y))$ has a unique solution given by the fixed point of $F_0.$ So, $F_2$ has no pseudo fixed points.  Next, we consider the partial order $\leq_\tau$ that aligns with the monotonicity of $F_2.$ In this case, $P_\tau=(y,y,y,y,y),$ and for each initial condition $X_0,$ we need to find $x<y$ such that $P_\tau\leq_\tau X_0,$ 
$$x<F_2(y,y,y)\quad \text{and}\quad  F_2(x,x,x)<y.$$
From the unique intersection, symmetry between the curves of $F_2(x,x,x)=y,$ $x=F_2(y,y,y)$ and the fact that 
$$F_2(0,0,0)=h(b^2+b+1),\qquad \lim_{x\to\infty}F_2(x,x,x)=h,$$
we obtain an unbounded solution set.  Therefore, the conditions of Theorem \ref{Th-GlobalStability} apply to $F_2$, and Theorem \ref{Th-GlobalStability-Main2} guarantees that $\bar x$ is globally attracting. 
\end{proof}

Finally, note that the one-dimensional map $g(x)=F_2(x,x,x)$ in Theorem \ref{Th-ClarkGlobalStability} is decreasing, and when $g$ and $g^{-1}$ have a unique intersection, we can guarantee that it has no $2$-cycle by assuming $g^\prime(\bar x)\geq -1.$ This condition translates into 
$$|f^\prime(\bar x)|\leq \frac{(2\bar x-h)(\bar x^2-h\bar x+h^2)}{\bar x^2(3\bar x^2-3h\bar x+h^2)}.$$
Interestingly, this is the same as the local stability condition obtained from Inequality \eqref{In-V2} when $a_2<-a_0^3$. 

\subsection{Clark's type models}
For species that need $k$-seasons to be recruited to the breeding population, Clark \cite{Cl1976} proposed the delay-difference equation
\begin{equation}\label{Eq-Clark}
x_{n+1}=ax_n+(1-a)f(x_{n-k}),
\end{equation}
where $0<a<1$  represents a survivorship coefficient, and $f$ denotes a stock-recruitment function constrained by specific species' characteristics \cite{Al1963,Fi-Go1984,Bo1992}. For instance, $f$ can be considered decreasing and of the generalized Beverton-Holt form or Ricker type, i.e., $f(t)=\frac{bt}{1+t^s},\; b>1,s>0,$ or  $f(t)=ct^se^{-t},\; c,s>0,$ respectively.  In a simple Cournot oligopoly model \cite{Ca2023,Th1960,Ga2009}, the characteristic polynomial of the Jacobian matrix is reduced to the one obtained from Eq. \eqref{Eq-Clark} when $f$ is chosen to be an affine map. Numerous scholarly articles have been published since then, exploring various aspects of its dynamics \cite{El-Li2005,El-Li2006,El-Lo-Li2008}. For more details, we refer the reader to a short survey by Liz \cite{Li2020} and the references therein. Our objective is not to examine the model and its dynamics extensively but to apply our expansion technique to this equation and showcase its simplicity and effectiveness in establishing certain outcomes. Readers are encouraged to form conclusions by comparing our results with other findings.
\\

Observe that a fixed point of $f$ is an equilibrium solution of Eq. \eqref{Eq-Clark}. So, we assume throughout this subsection that $f\in \mathcal{C}^1([0,\infty))$ and has at least one fixed point. Also, we assume that $f$ is positive and asymptotic to zero, i.e., $\lim_{t\to\infty} f(t)=0.$ At a fixed point $\bar x$ and based on our notations in Section \ref{Sec-Nonlinear},  we obtain the first row of the Jacobian matrix $J_0$ as
$$a_0=a,\; a_1=0,\;\ldots,\;a_{k-1}=0,\;a_{k}=(1-a)\beta,$$
where $\beta=f^\prime(\bar x),$
and its characteristic equation is
\begin{equation}\label{Eq-CharClark}
x^{k+1}-ax^k-(1-a)\beta=0.
\end{equation}
The local stability of Eq. \eqref{Eq-Clark} can be treated based on the results obtained in \cite{Pa1996,Ku1994} for equations of the form $x_{n+1}=ax_n+bx_{n-k}.$  However, our approach is more general, and we show how to tackle the local stability of Eq. \eqref{Eq-Clark} based on our expansion strategy.   We obtain $V_0=(a,0,(1-a)\beta),$ and consequently,
$$\|V_0\|_1=a+(1-a)|\beta|<1\quad \text{if}\quad |\beta|<1.$$
 This sufficient condition is the same as the one obtained in \cite{Cl1976} based on Rouche's theorem.  Next,  finding $V_1$ to $V_{k-1}$ does not help us to improve the result. The first improvement comes when finding $V_k.$ Indeed, we obtain
 \begin{equation}\label{Eq-Expansionk}
 \begin{split}
 x_{n+1}=&F_k(x_{n-k},\ldots,x_{n-2k})\\
 %%=&a^{k+1}x_{n-k}+(1-a)\sum_{j=0}^ka^jf(x_{n-k-j})\\
 =& a^{k+1}x_{n-k}+(1-a)f(x_{n-k})+(1-a)\sum_{j=1}^ka^jf(x_{n-k-j}),
 \end{split}
 \end{equation}
 and consequently,
\begin{align*}
V_k=&(a^{k+1}+\widetilde{\beta},a\widetilde{\beta},a^2\widetilde{\beta},a^3\widetilde{\beta},\ldots,a^k\widetilde{\beta})\\
V_{k+1}=&(a^{k+2}+2a\widetilde{\beta},a^2\widetilde{\beta},a^3\widetilde{\beta},\ldots,a^k\widetilde{\beta},(a^{k+1}+\widetilde{\beta})\widetilde{\beta}),
\end{align*}
where $\widetilde{\beta}=(1-a)\beta.$ Therefore,
\begin{equation}\label{Eq-Vk}
\begin{split}
\|V_k\|_1=&|a^{k+1}+(1-a)\beta|+a|\beta|(1-a^k)\\
\|V_{k+1}\|_1=&|a^{k+2}+2a(1-a)\beta|+a^2|\beta|(1-a^{k-1})+(1-a)|\beta||a^{k+1}+(1-a)\beta|.
\end{split}
\end{equation}
The boundary cases $a=0$ and $a=1$ can be tackled directly from Eq. \eqref{Eq-CharClark}; however, they are beyond our interest since $0<a<1.$ We focus on the stability region with respect to $\beta,$ and how it changes as we increase $k$.  At $\beta=1,$ we obtain $\|V_0\|_1=\|V_k\|_1=\|V_{k+1}\|_1=1,$ and this shows that we have some eigenvalues on the unit circle. Indeed, one of the eigenvalues is one at $\beta=1.$ Furthermore, $x^k(x-a)=\beta(1-a)$ has one solution larger than one for all $\beta>1.$ Therefore, $\bar x$ is unstable when $\beta>1,$ and it is locally stable if $|\beta|< 1$  regardless of the delay $k$.   Furthermore, the eigenvalues of the Jacobian matrix can be captured in the first step of our approach when $-1<\beta<1$. Another interesting fact that can be obtained from the norms in Systems \eqref{Eq-Vk} is that when $k\to \infty$, we obtain $\|V_k\|_1 \to |f^\prime(\bar x)|$ and
 $$\|V_{k+1}\|_1 \to |f^\prime(\bar x)|\left((1-a)^2|f^\prime(\bar x)|+a(2-a)\right),$$
 which is less than one when  $|f^\prime(\bar x)|<1.$ On the other hand, we can obtain the unit eigenvalues by substituting $x=e^{it}$ in Eq. \eqref{Eq-CharClark}. Indeed, this requires the solution of
 $$a=\frac{\sin((k+1)t)}{\sin(kt)}\quad \text{and}\quad \beta=\frac{\sin(t)}{\sin((k+1)t)-\sin(kt)}.$$
  Next, we need to focus on $\beta\leq -1.$ This turns out to be interesting since the stability region starts to be influenced by the delay $k.$
 In Fig. \ref{Fig-LocalStabilityClark}, we illustrate how the stability region shrinks when $k$ increases and the values of $(a,\beta)$ in which some eigenvalues are on the unit circle.
%%%%%%%%%%%%%%%%%%%%%%%%%%%%%%%%%%%%%%%%%%%%%%%%%%
\definecolor{MyMaroon}{rgb}{128,0,0}
\definecolor{MyOrange}{rgb}{255,87,51}
\definecolor{ffvvqq}{rgb}{1,0.3333333333333333,0}
\definecolor{qqqqff}{rgb}{0.,0.,1.}
\definecolor{ffqqqq}{rgb}{1.,0.,0.}
\definecolor{cqcqcq}{rgb}{0.7529,0.7529,0.7529}
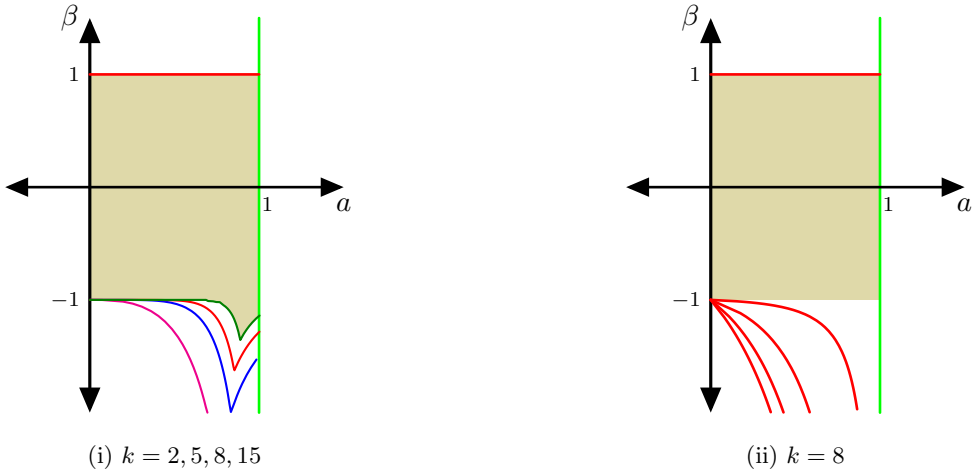
\begin{figure}[htbp]
\centering
\begin{minipage}[t]{0.5\textwidth}
%\raggedright
\begin{center}
\begin{tikzpicture}[line cap=round,line join=round,>=triangle 45,x=1.5cm,y=1.0cm,scale=1.5]
\fill[line width=1pt,color=ffqqqq,fill=olive!30,fill opacity=1.0] (0.01,-1)--(0.01,1)--(1,1)--(1,-1)--(0.01,-1);
\fill[line width=1pt,color=ffqqqq,fill=olive!30,fill opacity=1.0] (0.68668,-1.0)-- (0.69899,-1.01)-- (0.72883,-1.01276)-- (0.747,-1.0189)-- (0.768,-1.02)-- (0.789587,-1.0467)-- (0.8018,-1.060)-- (0.813,-1.07583)-- (0.824,-1.0949)-- (0.8355,-1.1196)-- (0.845,-1.145)-- (0.8526,-1.169)-- (0.858,-1.19)-- (0.8638,-1.2127)-- (0.869,-1.237)-- (0.874,-1.26309)-- (0.878,-1.284)-- (0.881,-1.304)-- (0.8838,-1.323)-- (0.8892,-1.35559)-- (0.9103,-1.29966)-- (0.926,-1.2645)-- (0.9415,-1.233)-- (0.9595,-1.20107)-- (0.983,-1.16494)-- (1.0,-1.139)--(1,-1)--(0.68668,-1)--cycle;
\draw[-triangle 45, line width=1.2pt,scale=1] (0,0) -- (0,1.5) node[left] {$\beta$};
\draw[line width=1.2pt,-triangle 45] (0,0) -- (0.0,-2);
%\draw[line width=0.8pt,color=green,dashed,domain=0:4.0] plot(\x,{\x});
%\draw[scale=1] (0.5,0.6) node[above,rotate=20] {\scriptsize $LC$};
%\draw[line width=1.0pt,color=red,smooth] (0,1)--(1,0);
\draw[line width=1.0pt,color=green,smooth] (1,-2)--(1,1.5);
\draw[line width=1.0pt,color=red,smooth] (0,1)--(1,1);
%\draw[line width=1pt,color=blue,smooth,samples=100,domain=0.02:5.0] plot(\x,{ln((\x*\x+4*\x)^0.5-\x)-ln(\x+(\x*\x+4*\x)^0.5)+\x/2+(\x*\x+4*\x)^0.5/2});
\draw[line width=0.8pt,color=magenta,smooth] (-3.497354434701949E-9,-0.9999999981885735)-- (0.07351019080633864,-1.00079477788826)-- (0.19611743557460687,-1.0152008114386022)-- (0.25325309741612434,-1.0330222399571054)-- (0.30751256882834044,-1.0599011288826217)-- (0.3625656434335355,-1.1000917427122106)-- (0.4061178265885558,-1.1435806859716773)-- (0.4441816916678101,-1.192107252496566)-- (0.4863119178116453,-1.2599187192965366)-- (0.521683043230978,-1.3309417736687985)-- (0.5481872404953287,-1.3944506271166457)-- (0.5702566558966768,-1.4553231761607703)-- (0.5932940228223353,-1.5279278774379632)-- (0.612534453002192,-1.5968022305294534)-- (0.631743802873977,-1.674258351057599)-- (0.647508832631303,-1.7452900623971574)-- (0.6609869749148147,-1.8120998843781924)-- (0.6748559662448519,-1.8874613624563694)-- (0.6854690144708248,-1.9501997170655234)-- (0.695,-2.0);
%\draw[scale=1] (0.2,-1) node[below,rotate=0] {\footnotesize $k=2$};
\draw[line width=0.8pt,color=blue,smooth]
(0,-1.000000000000115)-- (0.03473251060167802,-1.0000000052500813)-- (0.07862826722064908,-1.0000004714928503)-- (0.13481297067142603,-1.000012011420377)-- (0.17343685226325722,-1.0000544431911658)-- (0.2138189669809474,-1.000191143362965)-- (0.2542288714903642,-1.000540133883142)-- (0.2806143913945225,-1.0009770190173537)-- (0.2999586261051932,-1.0014578529417273)-- (0.336513297854909,-1.002908536350985)-- (0.3683909878207522,-1.0050115272919051)-- (0.39902509697546124,-1.0081056518171638)-- (0.4385811305033168,-1.0143361178662489)-- (0.4785741568560916,-1.0243206256002262)-- (0.5139292050177502,-1.0375427879602694)-- (0.5446384417131681,-1.0536001045096464)-- (0.5668746692552878,-1.0686448939691515)-- (0.6010422344211881,-1.0989538849628706)-- (0.6225088968790232,-1.1235781333821833)-- (0.6450663957069481,-1.155285661355257)-- (0.6588713559446273,-1.1781972372087788)-- (0.6714266296116292,-1.2017222990407517)-- (0.6889341042476627,-1.2394454519393925)-- (0.7073793644350019,-1.2864705060506902)-- (0.7228505363703254,-1.3327868018773161)-- (0.731233846740446,-1.3609278854656803)-- (0.7429760419006608,-1.4044497453413483)-- (0.7522338833554687,-1.4425480842732892)-- (0.7603824848113949,-1.4791827484636038)-- (0.76812759616881,-1.5169893386257498)-- (0.7756178781464103,-1.5566105716063174)-- (0.7824434756283567,-1.5956016288463666)-- (0.7881704682731028,-1.6306416339436693)-- (0.7922734657081489,-1.657154467601908)-- (0.7985553690358556,-1.7002087625026256)-- (0.803919018469837,-1.739516190726729)-- (0.8077450863879233,-1.769104365941183)-- (0.8114858877381304,-1.7993656004437335)-- (0.8148144381420028,-1.8274647232498664)-- (0.819054110170524,-1.8660118256095004)-- (0.8221047029773836,-1.8941534147977674)-- (0.8254238972289148,-1.9254071561309316)-- (0.8336351165980813,-1.9941289437585705)-- (0.8490081968657595,-1.9339457531016657)-- (0.8702066695284729,-1.84785494982338)-- (0.8850768474767399,-1.7941050761051227)-- (0.9003666469623707,-1.743554776006735)-- (0.927844104283466,-1.6630603657826457)-- (0.9527443065295107,-1.5998308858831465)-- (0.9846364883401938,-1.5288340192043872);
%\draw[scale=1] (0.2,-1) node[below,rotate=0] {\footnotesize $k=5$};
\draw[line width=0.8pt,color=red,smooth]
(0,-1)-- (0.1,-1.0000000020150075)-- (0.1999999999456642,-1.000001030192706)-- (0.2999999499613303,-1.0000393718653275)-- (0.3999935433598867,-1.0005243560758381)-- (0.4268440354996569,-1.0009411511865884)-- (0.4607791483774017,-1.001874448143621)-- (0.48404952631444387,-1.0029219464824972)-- (0.49971847153351945,-1.003894059179962)-- (0.5202917123718416,-1.0056036092184768)-- (0.5402599878738091,-1.0078733919670415)-- (0.5599996448159115,-1.0108912580286633)-- (0.5799526636117613,-1.0149554139033836)-- (0.5998957513674484,-1.020328431252612)-- (0.6191282691665132,-1.0270956675620058)-- (0.6382427406439495,-1.0357769374279768)-- (0.6586248300664973,-1.0477552383072744)-- (0.6773107836691136,-1.061851356920097)-- (0.6935908650358996,-1.0771614825912055)-- (0.7134272258570822,-1.100573847321687)-- (0.725286918680122,-1.1176072171310198)-- (0.7401208145727296,-1.142791898025008)-- (0.7545963432187837,-1.1723287755741365)-- (0.7627391457927573,-1.1914792299456651)-- (0.7692632414735814,-1.2083230996828762)-- (0.7793525199435011,-1.2373195426669525)-- (0.7852285719323769,-1.2560425406223859)-- (0.7939551505311397,-1.2866586626640757)-- (0.8009418889088982,-1.3138685033781088)-- (0.8090195077936468,-1.3487061763710724)-- (0.8151592718760867,-1.377898920821024)-- (0.8190475840747078,-1.3977235636117076)-- (0.8248370404435386,-1.4293404051445737)-- (0.828661916057029,-1.4517173993112302)-- (0.8318711866124742,-1.471478073608864)-- (0.8347660758673916,-1.4901213691739479)-- (0.838940140877559,-1.5184578616253588)-- (0.8428266035375951,-1.54649621566755)-- (0.8460608195460902,-1.5711356128231615)-- (0.8489088782039489,-1.593881537542332)-- (0.8545910986602482,-1.62412516042839)-- (0.8761195428845527,-1.5529126727295957)-- (0.8868375395354196,-1.5210476494637497)-- (0.8950387344291122,-1.4980846578505242)-- (0.9099640400466247,-1.459190928128181)-- (0.9234023706114811,-1.427092471791573)-- (0.9381659481444143,-1.3947012978013684)-- (0.9560201976313409,-1.3591346885413822)-- (0.9752200130003847,-1.32480023833262)-- (1.0024874256973024,-1.282084701807734);
%\draw[scale=1] (0.2,-1) node[below,rotate=0] {\footnotesize $k=8$};
\draw[line width=0.8pt,color=MyGreen,smooth]
(0,-1)--(0.68668,-1.0)-- (0.69899,-1.01)-- (0.72883,-1.01276)-- (0.747,-1.0189)-- (0.768,-1.02)-- (0.789587,-1.0467)-- (0.8018,-1.060)-- (0.813,-1.07583)-- (0.824,-1.0949)-- (0.8355,-1.1196)-- (0.845,-1.145)-- (0.8526,-1.169)-- (0.858,-1.19)-- (0.8638,-1.2127)-- (0.869,-1.237)-- (0.874,-1.26309)-- (0.878,-1.284)-- (0.881,-1.304)-- (0.8838,-1.323)-- (0.8892,-1.35559)-- (0.9103,-1.29966)-- (0.926,-1.2645)-- (0.9415,-1.233)-- (0.9595,-1.20107)-- (0.983,-1.16494)-- (1.003,-1.139);
%\draw[scale=1] (0.2,-1) node[below,rotate=0] {\footnotesize $k=15$};
\draw[-triangle 45, line width=1.0pt,scale=1] (0,0) -- (1.5,0) node[below] {$a$};
\draw[line width=1.0pt,-triangle 45] (0,0) -- (-0.5,0);
%\draw[scale=1] (4.5,3.5) node[below,rotate=39] {\footnotesize $b_\infty$};
\draw[scale=1] (0,1) node[left,rotate=0] {\scriptsize $1$};
\draw[scale=1] (0,-1) node[left,rotate=0] {\scriptsize $-1$};
\draw[scale=1] (1.05,0) node[below,rotate=0] {\scriptsize $1$};
\draw[scale=1] (0.5,-2.2) node[below,rotate=0] {\footnotesize (i) $k=2,5,8,15$};
\end{tikzpicture}
\end{center}
\end{minipage}%
%%%%%%%%%%%%%%%%%%%%%%%%%%%%%%%%%%%%%%%%%%%%%%%%%%%%%%%%%%%%%%%%%%%%%%%%%%%%%%%%%%%%%%%%%%%%
\begin{minipage}[t]{0.5\textwidth}
%\raggedright
\begin{center}
\begin{tikzpicture}[line cap=round,line join=round,>=triangle 45,x=1.5cm,y=1.0cm,scale=1.5]
\fill[line width=1pt,color=ffqqqq,fill=olive!30,fill opacity=1.0] (0.01,-1)--(0.01,1)--(1,1)--(1,-1)--(0.01,-1);
\draw[-triangle 45, line width=1.2pt,scale=1] (0,0) -- (0,1.5) node[left] {$\beta$};
\draw[line width=1.2pt,-triangle 45] (0,0) -- (0.0,-2);
%\draw[line width=0.8pt,color=green,dashed,domain=0:4.0] plot(\x,{\x});
%\draw[scale=1] (0.5,0.6) node[above,rotate=20] {\scriptsize $LC$};
%\draw[line width=1.0pt,color=red,smooth] (0,1)--(1,0);
\draw[line width=1.0pt,color=green,smooth] (1,-2)--(1,1.5);
\draw[line width=1.0pt,color=red,smooth] (0,1)--(1,1);
\draw[line width=1.0pt,color=red, smooth] (0.0,-1.0)-- (0.15302037601835744,-1.0123066364473206)-- (0.19318086628995057,-1.0168780247872962)-- (0.2337798478179165,-1.0222883737053854)-- (0.27398384095657774,-1.028590399640639)-- (0.31391459493800117,-1.0359812152443209)-- (0.3465642614352677,-1.0430378438345718)-- (0.3758965308440974,-1.0503068356487089)-- (0.41959511000844335,-1.0631186753327164)-- (0.45742373585755464,-1.0765826048835414)-- (0.4877081633357727,-1.0893471810574695)-- (0.5200308314403673,-1.105381169168129)-- (0.5592336944427969,-1.12903190805281)-- (0.5862905402575693,-1.1487502066408706)-- (0.6256830098797552,-1.1839422905848598)-- (0.6538447123489637,-1.215199759499137)-- (0.6751565376058,-1.243217632033639)-- (0.6979658289777819,-1.27841784376568)-- (0.7216719326164582,-1.3221685415322828)-- (0.7375300874451692,-1.3565177811236264)-- (0.7572335494104896,-1.4063000635289848)-- (0.7747424642555346,-1.458739151403977)-- (0.7867689181683496,-1.500283548344641)-- (0.7991233411607145,-1.5486401189079797)-- (0.8075274109052002,-1.5853922080464713)-- (0.8169814461452439,-1.6310930952510694)-- (0.825932061274006,-1.6792735912118162)-- (0.8335097303563075,-1.7243874167166726)-- (0.841809515659323,-1.7790694792072437)-- (0.8477174117098971,-1.8218325244287932)-- (0.8530881544586214,-1.8638499352063433)-- (0.8572382415172852,-1.8985901700368486)-- (0.8608962234559384,-1.9310100747865317)-- (0.8650429132471792,-1.9699811931187505);
\draw[line width=1.0pt,color=red, smooth] (0,-1.0)-- (0.17529677628013313,-1.117883795806684)-- (0.20889211760269952,-1.1494160816816579)-- (0.2481176120213297,-1.1911803734999205)-- (0.2777748118616203,-1.226830345209673)-- (0.3099704552461854,-1.2701163859377305)-- (0.3385932068303306,-1.3131806378720778)-- (0.3654782923571062,-1.35812314651868)-- (0.39858318795978576,-1.4203525855562975)-- (0.4281445860668694,-1.4833825484955732)-- (0.4508331582358956,-1.5372927043127818)-- (0.47627013716168864,-1.6043119912633326)-- (0.4944862818244733,-1.657149312141688)-- (0.511969474875131,-1.7121441557005384)-- (0.5290325882242902,-1.770321294366351)-- (0.5470964317561615,-1.8373267090889216)-- (0.5570893073849901,-1.8770328013942923)-- (0.5675923353586044,-1.9209739593125468)-- (0.5763013737721769,-1.9592448104453308)-- (0.5844075796798777,-1.9964603742219766);
\draw[line width=1.0pt,color=red, smooth] (0.0,-0.999)-- (0.10852525203705027,-1.1477814333385774)-- (0.1402037161628099,-1.1998280752562431)-- (0.17001286056077494,-1.253208951551883)-- (0.19832339612184044,-1.3082909428903444)-- (0.21979176569718167,-1.3531909117350354)-- (0.2467538595055279,-1.413785174272547)-- (0.26618372859133355,-1.4606158598107024)-- (0.28132381021517805,-1.4990981840917954)-- (0.29747177985110607,-1.542200123639997)-- (0.3112018902513082,-1.5806262027578721)-- (0.32456718223142844,-1.6196992787590148)-- (0.3376266160639218,-1.6595622951600872)-- (0.3490821863452161,-1.6959785330860013)-- (0.3615655273189545,-1.7372916460759258)-- (0.38116250322685047,-1.805812927358851)-- (0.40073233521954804,-1.879088359704064)-- (0.41227323431995433,-1.924765783587224)-- (0.4217388266155201,-1.9636889730159435)-- (0.42983210751260487,-1.9980648801210479);
\draw[line width=1.0pt,color=red, smooth] (0.0,-1.0)-- (0.08135756472127313,-1.1575107099994493)-- (0.10672806685008761,-1.2129344461951286)-- (0.12560744706355706,-1.2563930899260765)-- (0.1500837817584809,-1.315770456639924)-- (0.1690047051904278,-1.3641908853218652)-- (0.18671692020393807,-1.411655682273267)-- (0.20464984186662227,-1.4619593084709162)-- (0.2205027986510913,-1.5084336470213784)-- (0.23671181444301548,-1.5580227646129334)-- (0.2507004338526757,-1.6026041329598133)-- (0.26493388148633235,-1.6497645887505639)-- (0.28046774925115087,-1.7034287770930165)-- (0.29194977913557846,-1.7446525135087838)-- (0.3044039586162196,-1.7909471022051209)-- (0.3151702925611788,-1.8323596426647784)-- (0.3235136985320287,-1.8653812280779354)-- (0.3326347329574575,-1.902447375068821)-- (0.3444307982982756,-1.9519482586724428)-- (0.3545,-1.99579);
\draw[-triangle 45, line width=1.0pt,scale=1] (0,0) -- (1.5,0) node[below] {$a$};
\draw[line width=1.0pt,-triangle 45] (0,0) -- (-0.5,0);
%\draw[scale=1] (4.5,3.5) node[below,rotate=39] {\footnotesize $b_\infty$};
\draw[scale=1] (0,1) node[left,rotate=0] {\scriptsize $1$};
\draw[scale=1] (0,-1) node[left,rotate=0] {\scriptsize $-1$};
\draw[scale=1] (1.05,0) node[below,rotate=0] {\scriptsize $1$};
\draw[scale=1] (0.5,-2.2) node[below,rotate=0] {\footnotesize (ii) $k=8$};
%\draw[scale=1] (1.5,3.0) node[above,rotate=30] {\footnotesize Instability};
\end{tikzpicture}
\end{center}
\end{minipage}%
\caption{Part (i) of this figure shows the stability we obtain based on our expansion strategy and the parameter values. It is done according to $\|V_k\|_1<1.$ The lower part of the region shrinks as we increase $k$ ($k=2,5,8,15$).  Part (ii) of the figure shows the curves in which we obtain eigenvalues on the unit disk for the case $k=8$. In this case, the upper curve forms the lower boundary of the region (see Proposition 3 in \cite{El-Lo-Li2008}).  A comparison between the two graphs shows the effectiveness of our approach.}\label{Fig-LocalStabilityClark}
\end{figure}
%%%%%%%%%%%%%%%%%%%%%%%%%%%%%%%%%%%%%%%%%%%%%%%%%%%%%%%%%%%%%%%%%%%%%%%%%%%%%%%%%%%%%%%%%%%%%%%%%%%%%%%
\\

 Now, we proceed to discuss global stability. The maps $F_0$ in Eq. \eqref{Eq-Clark} and $F_k$ in Eq. \eqref{Eq-Expansionk} have the same fixed points, and we assume uniqueness.

 \begin{proposition}\label{Pr-Clark}
  Consider $a$ and $f$ as given in Eq. \eqref{Eq-Clark}, and assume $f$ has a unique positive fixed point $\bar x$. Define 
  $$g_2(x)=-\frac{\gamma_1+1}{1-\gamma_1}x+\frac{2}{1-\gamma_1}f(x)\quad\text{where}\quad \gamma_1=-\frac{1+a^{k+1}}{1-a^{k+1}}.$$
  Each of the following holds true:
  \begin{description}
 \item{(i)} If $\bar x$ is locally stable for $f,$ then it is locally stable for $g_2.$ 
 \item{(ii)}  If $\bar x$ is locally stable for $f$ and $f$ is increasing, then $\bar x$ is globally stable for $g_2$ with respect to the domain $(0,\infty).$ 
\item{(iii)} If $g_2^\prime(t)\neq -1,$  then $g_2$ has no $2$-cycle.
 \end{description}
 \end{proposition}
 \begin{proof}
 Part (i): Based on our definition of local stability, we have $-1<f^\prime(\bar x)<1.$ Also, observe that $\gamma_1<-1.$ Therefore, we obtain 
 $$1>g_2^\prime(\bar x)>\frac{3+\gamma_1}{\gamma_1-1}>-1.$$
 Therefore, $|g_2^\prime(\bar x)|<1$ and $\bar x$ is LAS.  Part (ii) follows from the fact that $(x-\bar x)(g_2( x)- x)<0$ for all $x>0.$ To clarify Part (iii), observe that to have a $2$-cycle, say $\{x,y\},$ we need $y=g_2(x)$ and $x=g_2(y),$ which gives us
 $$ \gamma_1=\frac{f(x)-f(y)}{x-y}=f^\prime(t),$$
 for some $t$ between $x$ and $y$. However, $f^\prime(t)=\gamma_1$ means  $g_2^\prime(t)= -1$, which is impossible by the given condition.  
 \end{proof}
Note that the scenario in which $f$ is increasing is straightforward to tackle and can be accomplished based on $F_0$ instead of $F_k$. The one-dimensional map $g(t)=at+(1-a)f(t)$ can be used to establish global stability. We shall disregard this case. One more technical result is needed before we give our global stability theorem.
 \begin{proposition}\label{Pr-Clark2}
Consider $a$ and $f$ as given in Eq. \eqref{Eq-Clark}, and assume $f$ is decreasing with a unique positive fixed point $\bar x$. Define 
$$g_3(x)=\frac{\gamma_2 x-f(x)}{\gamma_2 -1}\quad\text{where}\quad \gamma_2=\frac{1-a^{k}}{1-a}.$$
Each of the following holds true:
  \begin{description}
 \item{(i)} $g_3$ is increasing with a unique fixed point at $\bar x$ and a unique positive zero $x^*$ such that $0< x^*<\bar x.$
 \item{(ii)}  The map $g_4(x)=f^{-1}(g_3(x))$ is well defined on $(x^*,g_3^{-1}(f(0))]\to [0,\infty)$, and $g_4$ is decreasing.
 \end{description}
 \end{proposition}
 \begin{proof}
 (i) Part (i) is obvious since $\gamma_2>1$ and $f$ is decreasing. In Part (ii), the domain and codomain of $g_4$ are straightforward to check. Also, $g_4$ is decreasing because $g_3$ is increasing and $f^{-1}$ is decreasing. 
\end{proof}
 \begin{theorem}\label{Th-ClarkGlobalStability}
 Consider Eq. \eqref{Eq-Clark} and its expansion in Eq. \eqref{Eq-Expansionk}. Assume $f$ is decreasing and has a unique positive fixed point $\bar x$. Let $g_2$  be as defined in Proposition \ref{Pr-Clark}, $g_3$ and $g_4$ as defined in Proposition  \ref{Pr-Clark2}, and
 $$\gamma:=-\frac{a^{k+1}}{1-a}.$$ 
 The unique equilibrium solution of Eq. \eqref{Eq-Clark} is globally attracting in each of the following cases:
 \begin{description}
 \item{(i)}  If $\sup f^\prime(t)<\gamma$ and $f^\prime(t)\neq \gamma_1,$ 
 \item{(ii)} If $\inf f^\prime(t)>\gamma$ and $g_4$ has no $2$-cycles.
 \end{description}
 \end{theorem}
 \begin{proof}
 \begin{description}
 \item{(i)} Assume $\sup f^\prime(t)<\gamma,$ then $F_k$ is decreasing in each of its arguments. We consider the $\leq_\tau$ order that aligns with the monotonicity of $F_2$; in this case, we have $P_\tau=(y,y,\ldots,y).$ We focus on the feasible solution of the inequalities 
 $$x<y,\quad F_k(x,\ldots,x)=g_2(x)<y\quad \text{and}\quad x<F_k(y,\ldots,y)=g_2(y).$$
 From Proposition \ref{Pr-Clark}, $f^\prime(t)\neq \gamma_1$ guarantees that $g_2$-cycle. So, $g_2$ is decreasing, and the curves $y=g_2(x)$ and its inverse (the reflection about $y=x$) do not intersect at any point other than $\bar x$. This makes the set of feasible solutions for the three inequalities unbounded inside $\{(x,y): 0<x<\bar x,y\geq \bar x_2\}$. Therefore, for each initial condition $X_0=(x_{-2k},\ldots,x_0)$, there exists $P_\tau$ such that the conditions of Theorem \ref{Th-GlobalStability-Main2} are satisfied. 
 \item{(ii)} Assume $\inf f^\prime(t)>\gamma,$ then $F_k$ is increasing in the $x_{n-k}$ argument and decreasing in the others. Again, consider the $\leq_\tau$ order that aligns with the monotonicity of $F_k$, and in this case, we have $P_\tau=(x,y,\ldots,y).$ To exclude the existence of pseudo-fixed points, solve 
 $$(x,y)=(F_k(x,y,\ldots,y),F_k(y,x,\ldots,x)).$$
This means that $f(x)=g_3(y)$ and $f(y)=g_3(x),$ where $g_3$ is given in Proposition \ref{Pr-Clark2}. Because $g_4$ has no $2$-cycles, the only solution for the system of equations is $x=y=\bar x,$ which means that $F_k$ has no pseudo-fixed points.
 
 Next, we focus on the feasible solution of the inequalities 
 $$x<y,\quad x<F_k(x,y,\ldots,y)\quad \text{and}\quad F_k(y,x,\ldots,x)<y.$$
This simplifies to 
  \begin{equation}\label{In1-Clark}
 x<y,\quad f(x)<g_3(y) \quad \text{and}\quad f(y)>g_3(x). 
  \end{equation}
When $x\leq x^*\leq y$, then inequalities $x<y$ and $f(y)>g_3(x)$ are valid by default. Also, because $f$ is decreasing, $g_3$ is increasing, and they intersect at $\bar x,$ then $f(x)<g_3(y)$ is also valid. So, we proceed to focus on $x>x^*.$ Since $f$ is decreasing, $g_4=f^{-1}g_3$ becomes decreasing, and we can reduce the inequalities in \eqref{In1-Clark} to
 $$x^*< x\leq \bar x,\quad y\geq \bar x,\quad x>g_4(y)\quad \text{and}\quad y<g_4(x). $$
 $g_4$ has a vertical asymptote at $x=x^*$ and an $x$-intercept at $x=g_3^{-1}(f(0))>x^*$. This gives an unbounded region between $y=g_4(x)$ and its inverse $x=g_4(y)$. Fig. \ref{Fig-ClarkExample} gives a visual illustration of the feasible region obtained for the solution of the three inequalities in \eqref{In1-Clark}. Finally, as in Case (i), we verified that for each initial condition $X_0=(x_{-2k},\ldots,x_0)$, there exists $P_\tau$ such that the conditions of Theorem \ref{Th-GlobalStability-Main2} are satisfied, which completes the proof.  
 \end{description}
 \end{proof}
 
It is worth mentioning that Theorem 2.4.1 in \cite{Ko-La1993} states that the unique equilibrium of Eq. \eqref{Eq-Clark} is globally stable if $f$ is decreasing and has no $2$-cycle. However, Theorem \ref{Th-GlobalStability-Main2} goes beyond this result, as we illustrate by taking the simple affine map $f(t)=-t+2b.$ Since $f$ has infinitely many $2$-cycles, Theorem 2.4.1 in \cite{Ko-La1993} fails to address this case. However, Theorem \ref{Th-ClarkGlobalStability} is valid here. Indeed, $f^\prime(t)<\gamma_2$ is valid for all values of $a$ that satisfy $a(1+a^k)<1.$ Also, $g(t)=(2a^{k+1}-1)t+2b(1-a^{k+1})$ has no  $2$-cycle for all values of $0<a<1.$ So, Part (i) of the Theorem is applicable here. To illustrate Part (ii), we give our final example.

\begin{example}
Consider Eq. \eqref{Eq-Clark} with $f(t)=\frac{b}{1+t},\; t\geq 0$ such that $b<-\gamma=\frac{a^{k+1}}{1-a}.$ In this case, $f$ is decreasing and $\inf f^\prime (t)>\gamma.$  
We fix $b=2,$ $k=3$ and $a=\frac{7}{10}.$ In reference to Proposition \ref{Pr-Clark2}, we have
$$\gamma_2=\frac{219}{100}\quad \text{and}\quad g_3(x)=\frac{219}{119}x - \frac{200}{119(x + 1)}.$$ 
Also,
$$g_4(x)=\frac{438+19x-219x^2}{219x^2 + 219x - 200}.$$
The curves and the feasible region of the inequalities in \eqref{In1-Clark} are plotted in Fig. \ref{Fig-ClarkExample}
\end{example}

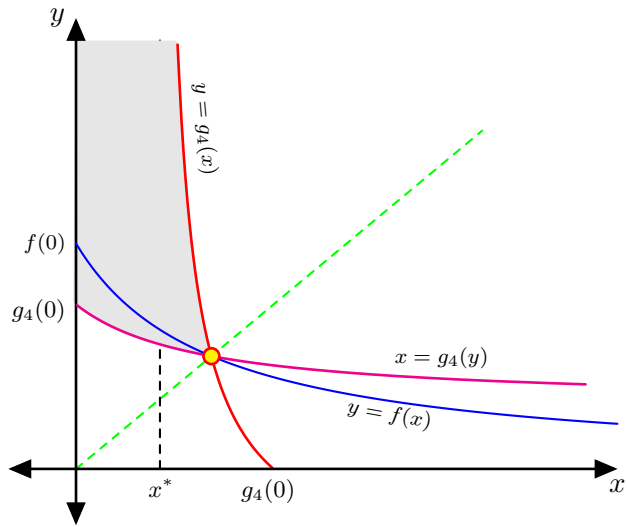
\begin{figure}[htbp]
\begin{center}
\begin{tikzpicture}[line cap=round,line join=round,>=triangle 45,x=1.2cm,y=1.0cm,scale=1.5]
\draw[line width=0.8pt,color=green,dashed,domain=0:3.0] plot(\x,{\x});
\draw[line width=0.8pt,color=black,dashed] (0.62,0)--(0.62,3.8);
\fill[fill=gray!20] (0,3.8) -- (-0.021076570638651364,1.4786061035162508) -- (0.05478844855663731,1.4093498215215403) -- (0.12756838355029654,1.3520771559281783) -- (0.21600968850938068,1.2921328816944395) -- (0.3006806386118561,1.2427310830440725) -- (0.3954896177679382,1.1948488829712254) -- (0.48088584032056647,1.157202259367219) -- (0.5650494960630458,1.1242955518250524) -- (0.6626591772581826,1.0904833104613982) -- (0.7705137054648077,1.0576393796104706) -- (0.8791611604844384,1.0285215872768227) -- (1,1) -- (0.9748148148148159,1.121039609711936) -- (0.9511111111111121,1.2497902705036725) -- (0.8977777777777788,1.6088791466971815) -- (0.88,1.7566390000591452) -- (0.8622222222222232,1.9228239877008502) -- (0.8503703703703713,2.0456342764436455) -- (0.8385185185185196,2.17959432625931) -- (0.8266666666666677,2.3263010993136293) -- (0.8088888888888899,2.5745634967885156) -- (0.7614814814814824,3.471695998316672) -- (0.7437037037037046,3.8) -- cycle;
\draw[-triangle 45, line width=1.0pt,scale=1] (0,0) -- (0,4.0) node[left] {$y$};
\draw[line width=1.0pt,-triangle 45] (0,0) -- (0.0,-0.5);
\draw[line width=0.8pt,color=blue,smooth,samples=100,domain=0.0:4.0] plot(\x,{2/(1+\x)});
\draw[line width=1pt,color=magenta,smooth,variable=\x,domain=0.75:1.45] plot({-(219*\x*\x - 19*\x - 438)/(219*\x*\x + 219*\x - 200)},\x);
\draw[line width=1pt,color=red,smooth,variable=\x,domain=0.75:1.45] plot(\x,{-(219*\x*\x - 19*\x - 438)/(219*\x*\x + 219*\x - 200)});
\draw[-triangle 45, line width=1.0pt,scale=1] (0,0) -- (4.0,0) node[below] {$x$};
\draw[line width=1.0pt,-triangle 45] (0,0) -- (-0.5,0);
\draw[scale=1] (0,2) node[left,rotate=0] {\scriptsize $f(0)$};
\draw[scale=1] (2.7,0.4) node[left,rotate=-10] {\scriptsize $y=f(x)$};
\draw[scale=1] (2.7,0.8) node[above,rotate=0] {\scriptsize $x=g_4(y)$};
\draw[scale=1] (0.8,3) node[above,rotate=-85] {\scriptsize $y=g_4(x)$};
\draw[scale=1] (1.42,0) node[below,rotate=0] {\footnotesize $g_4(0)$};
\draw[scale=1] (0.63,0) node[below,rotate=0] {\footnotesize $x^*$};
\draw[scale=1] (0,1.4) node[left,rotate=0] {\footnotesize $g_4(0)$};
\draw[line width=1.0pt,red,fill=yellow] (1.0,1.0) circle (2.0pt);
\end{tikzpicture}
\end{center}
\caption{This figure shows the feasible region of the inequalities in \eqref{In1-Clark} }\label{Fig-ClarkExample}
\end{figure}
\section{Conclusion and discussion}
We considered the general scalar $k$-dimensional difference equation $x_{n+1}=F_0(x_n,\ldots,x_{n-k})$, then we expanded the difference equation by successive substitutions. This gives us a system of the form $y_{n+1}=F_m(y_{n-m},\ldots,y_{n-m-k})$ that has a delay $m+k$, and we denote this process by the expansion strategy. Every solution of the original system can be a solution of the expanded system, but the converse is not necessarily true. We investigated the stability connection between the two systems. The expansion process has advantages that are of paramount significance in stability analysis. When the map $F_0$ is affine, the expansion technique leads to a local stability algorithm that can be used as an alternative to the Jury stability algorithm. The results here expand and complement the results obtained in \cite{Al-Ca-Ka2025}. In particular, the results in \cite{Al-Ca-Ka2025} establish sufficient conditions, while here, we establish the necessary and sufficient conditions.  The conclusion of the algorithm states the following: Let the characteristic polynomial of the Jacobian matrix $J_0$ of $F_0$ at an equilibrium solution be
$$p_0(x)=x^k+a_0x^{k-1}+\cdots,a_{k-2}x+a_{k-1}.$$
Define the initial column vector $V_0=(a_0,a_1,\ldots,a_{k-1})^t,$ then the vectors $V_m=\left(J_0^t\right)^mV_0.$ The equilibrium is locally stable if and only if $\|V_m\|_1<1$ for some finite number $m$.
\\

When our expansion strategy is used for a nonlinear system $x_{n+1}=F_0(x_n,\ldots,x_{n-k+1})$, local stability can be achieved through the same results obtained in the linear case. In other words, if $F_0$ is sufficiently smooth, a hyperbolic fixed point of $F_0$ is locally asymptotically stable if and only if $\|\nabla F_m\|_1<1$ at the fixed point for some finite number $m.$. Furthermore, the expanded map $F_m$ can have different monotonicity characteristics from those of $F_0$. In particular, when $F_m$ becomes monotonic in each argument, the power of the embedding technique of \cite{Al-Ca-Ka2024B} can be utilized to prove global stability. Regardless of the technique you use for global stability, the rule of thumb becomes as follows: if you obtain local stability from $\|\nabla F_m\|_1<1,$ then use the expanded system $F_m$ to tackle global stability. Several examples, including Ricker's and Clark's models, illustrate the expansion strategy and its effectiveness when combined with the embedding technique. 
\bigskip

%-----------------------------------
\noindent{\textbf{Acknowledgement:}} The authors thank the reviewers for their comments and remarks that lead to the current version of our paper. The first author is supported by a sabbatical leave from the American University of Sharjah and by Maria Zambrano grant for attracting international talent from the Polytechnic University of Cartagena.
\bibliographystyle{unsrt}

\bibliography{AlSharawi-bibliography}
\end{document}